\documentclass[10pt]{amsart}

\usepackage{amssymb,amsthm,amsmath,enumerate}
\usepackage[numbers,sort&compress]{natbib}
\usepackage{color}
\usepackage{graphicx}
\usepackage{amssymb,amsthm,amsmath}
\usepackage[numbers,sort&compress]{natbib}
\usepackage{color}
\usepackage{graphicx}

\title[Embedded eigenvalues in spectral bands]{  Sharp spectral transition for eigenvalues embedded into the spectral bands of perturbed periodic operators}

\author{Wencai Liu}
\address[Wencai Liu]{Department of Mathematics, University of California, Irvine, California 92697-3875, USA}\email{liuwencai1226@gmail.com}

\author{ Darren C. Ong}
\address[  Darren C. Ong]{Department of Mathematics, Xiamen University Malaysia, Jalan Sunsuria, Bandar Sunsuria, Sepang, 43900, Selangor, Malaysia}
\urladdr{https://dongcl.wixsite.com/darrenong}
\email{darrenong@xmu.edu.my}

\theoremstyle{plain}
\newtheorem{theorem}{Theorem}[section]
\newtheorem{corollary}[theorem]{Corollary}
\newtheorem{lemma}[theorem]{Lemma}
\newtheorem{proposition}[theorem]{Proposition}

\newcommand{\R}{\mathbb{R}}
\newcommand{\Z}{\mathbb{Z}}
\theoremstyle{definition}

\newtheorem{remark}[theorem]{Remark}

\def\dco#1{\textcolor{black}{#1}}
\def\liu#1{\textcolor{black}{#1}}
\begin{document}


\begin{abstract}
In this paper,
we consider  the Schr\"odinger  equation,
\begin{equation*}
    Hu=-u^{\prime\prime}+(V(x)+V_0(x))u=Eu,
\end{equation*}
where $V_0(x)$ is 1-periodic and $V (x)$ is a decaying  perturbation.
By Floquet theory,  the spectrum
of $H_0=-\nabla^2+V_0$ is purely absolutely continuous and
consists of a union of closed intervals (often referred to as spectral bands).
Given any finite set of points  $\{ E_j\}_{j=1}^N$ in any spectral band of $H_0$ obeying a mild non-resonance condition,
we construct smooth  functions $V(x)=\frac{O(1)}{1+|x|}$ such that
$H=H_0+V$ has eigenvalues $\{ E_j\}_{j=1}^N$.
Given any countable set of  points  $\{ E_j\}$ in any spectral band of $H_0$ obeying the same non-resonance condition,
and any function $h(x)>0$ going  to infinity arbitrarily slowly,
we construct smooth  functions $|V(x)|\leq \frac{h(x)}{1+|x|}$ such that
$H=H_0+V$ has eigenvalues $\{ E_j\}$. On the other hand, we show that there is no eigenvalue of $H=H_0+V$ embedded in the spectral bands if  $V(x)=\frac{o(1)}{1+|x|}$ as $x$ goes to infinity. We prove also an analogous result for Jacobi operators.

\end{abstract}
\maketitle
\section{Introduction}

In this paper, we consider  the Schr\"odinger  equation,
\begin{equation}\label{Gu}
    Hu=-u^{\prime\prime}+(V(x)+V_0(x))u=Eu,
\end{equation}
where $V_0(x)$ is 1-periodic and $V (x)$ is a decaying  perturbation.

When $V\equiv 0$, we have an unperturbed $1$-periodic Schr\"odinger equation,
\begin{equation}\label{GV}
    H_0\varphi=-\varphi^{\prime\prime}+V_0(x)\varphi=E\varphi.
\end{equation}
We also consider a Jacobi eigenvalue equation, 
\begin{equation}\label{Jeq:JacobiUnperturbed}
  (J_0u)(n):=  a_{n+1} u({n+1})+a_n u({n-1})+b_{n+1}u(n)=Eu(n), n\geq 0,
\end{equation}
 where the $\{a_j, b_j\}$ are real sequences indexed by $j\geq 1$ with $a_j$ assumed to be positive.
Alternatively, we can view this eigenvalue equation in terms of a operator on $\ell^2(\mathbb Z_{\geq 0})$.
 We also consider perturbations of this equation, namely,
 \begin{equation}\label{Jeq:JacobiPerturbed}
  (Ju)(n)=  (a_{n+1}+a'_{n+1}) u({n+1})+(a_n+a_n') u({n-1})+(b_{n+1}+b'_{n+1})u(n)=Eu(n), n\geq 0,
 \end{equation}
 where $a'_j$ and $b'_j$ are real sequences chosen so $a_j+a'_j$ is always positive. Let  us assume in addition that the $a_j$ and $b_j$ sequences are periodic with period $q\geq 1$.

The present paper is the combination of our two preprints \cite{liu2018sharp1} and \cite{liu2018sharp}. These two preprints are not intended for publication.

Through basic Floquet theory, we know that the essential spectrum of the operators $H_0$ and $J_0$ both consist of absolutely continuous bands. Our goal is to identify perturbations that leave the absolutely continuous spectrum unchanged, but also produce embedded singular spectrum in these absolutely continuous bands.

 This is a problem with a long history. Let us consider first a special case, the free Schr\"odinger operator (that is, the operator $H_0$ in the case where $V_0\equiv 0$). Here the absolutely continuous spectrum is the interval $[0,\infty)$. For this operator, the classical Wigner-von Neumann result \cite{von1929uber} introduces a decaying oscillatory perturbation that produces a single embedded eigenvalue at $E=1$. Following this, it has been an enduring topic of interest in inverse spectral theory to find perturbations of the free operator that produce embedded point spectrum in $[0,\infty)$: see for instance \cite{jl1,jitomirskaya2018noncompact,KRS, KrugerJMAA12,LukicJST13,LukicCMP14,naboko1986dense,kiselev2005imbedded,simon1997some,RemlingPAMS2000,agmon1975}. See also \cite{DenisovKiselev} for a more detailed survey of results in this area.

A natural next step is to understand how to produce embedded point spectrum when $V_0 \not\equiv 0$. This more general problem has attracted recent interest \cite{LukicOngTAMS,Simonov2016,KRS,LotoreichikSimonov}. In addition, there has also been work done in embedded point spectrum for the spectral bands of other periodic operators, such as the Jacobi operator \cite{JudgeNabokoWood2016,nabokostu18,lukdcmv} and the CMV operator (\cite{lukdcmv} and \cite[Section 12.2]{SimonOPUC2}).

Our paper's main thrust may be summarized as follows. Let $V_0(x)$ be any $1$-periodic potential function, and consider any countable set $S$ embedded in a band of the essential spectrum of $H_0$ in \eqref{GV}. If $S$ satisfies a mild non-resonance condition, we then carefully construct a perturbation $V$ of $V_0$ so that the essential spectrum remains unchanged, and eigenvalues appear at every point in $S$. In other words, for a given band we can find a perturbation that can produce any embedded point spectrum we desire, as long as our set of eigenvalues obeys that weak non-resonance condition.

Our choice of perturbation is inspired by the one introduced in \cite{jitomirskaya2018noncompact}. Of course, since we are perturbing a periodic operator rather than a free operator the construction is different, and in many ways much more challenging. Rather than using the standard Pr\"ufer variables, we have to instead use the generalized Pr\"ufer variables introduced in \cite{KRS}, which are a lot more complicated. The main contribution of this paper is in Section  \ref{Section:Prep} where we have to perform several precise estimates on these generalized Pr\"ufer variables. One key innovation in this section is the use of a Fourier expansion to ensure that some key terms in our construction decay sufficiently quickly. After the Fourier expansion, we end up having to bound some decaying oscillatory functions, and we accomplish this by carefully ensuring that the positive parts and the negative parts of the decaying oscillations cancel out well enough.
The ideas in Section \ref{Section:Prep} are all new, and it is perhaps the most technically complicated part of our paper. We remark that the free perturbation setting explored in \cite{jitomirskaya2018noncompact} does not contain the obstacles we have to overcome here in Section \ref{Section:Prep}. \liu{ Actually, our result implies the almost orthogonalization  of  generalized Pr\"ufer angles
in a suitable Hilbert space, which allows us to investigate the distribution of  embedded eigenvalues \cite{liu2018asymptotical}. We believe our
  analysis provides  a useful  tool to tackle other topics in  the spectral theory of perturbed periodic operators.}

Our construction is an improvement over previous results in a few important ways. For example, the construction in Theorem 4 of  \cite{LukicOngTAMS} only produces a single embedded eigenvalue in each band.    In \cite{KRS}, Theorem 4.2 we are presented with a construction that can produce dense embedded point spectrum, but only if the desired eigenvalues satisfy a rational independence condition.
 The reason for these technical restrictions in previous results is that while it is not too  difficult to control the growth of the formal eigenfunction for one eigenvalue, simultaneously dealing with multiple eigenvalues at once is problematic. Point spectra are in a sense very fragile, so modifying a perturbation $V(x)$ to produce one eigenvalue often destroys the other eigenvalues. Thus simultaneously producing two embedded eigenvalues in a band is challenging, let alone infinitely many. We were able to overcome this problem   by making very careful choices in our construction of $V$.

We do admit a technical restriction on $S$, a non-resonance condition. Each point of every spectral band is assigned a \em quasimomentum \em, which is a phase parameter in $[0,\pi)$ related to the Floquet solution of the unperturbed periodic operator equation \eqref{GV}. Given any two points in $S$, we require that their quasimomenta not sum to $\pi$. This is a very natural condition that appears almost universally in the embedded eigenvalues literature. For example, in \cite{LukicOngTAMS} this non-resonance condition is addressed in their Lemma 13 (expressed as a condition on Fourier coefficients). In \cite{JanasSimonov} this condition is described as the complement of energies $\{\pm 2\cos (\omega), \pm 2\cos (2\omega)\}$. We emphasize that our condition is a much weaker than that the restriction in \cite[Theorem 4.2]{KRS}, which requires the set of quasimomenta to be rationally independent with each other and with $\pi$. In particular, if we restrict ourselves to half of the spectral band (e.g., the half of the band corresponding to quasimomenta in $(0,\pi/2)$) we can allow $S$ to be a completely arbitrary countable set.

Furthermore, by carefully tweaking our construction, we are able to ensure that our perturbation $V(x)$ can be made to a smooth function. This smoothness is known to be difficult to achieve even for the case when $V_0\equiv 0$. We are able to ensure smoothness due to the iterative nature of our construction, which allows us to make small, precise adjustments to the $V(x)$ function at each step to make it smooth, while still controlling the size of all the eigenfunctions.

With regard to the Jacobi versions of our  result, we remark that ours is a  very significant improvement over previous results in the literature. Eigenvalues are in a sense very fragile, and so forcing multiple embedded eigenvalues to appear simultaneously is often challenging. Compare for instance the result in \cite{nabokostu18}, which introduces a perturbation that can only produce two embedded eigenvalues.  \liu{In another very  recent paper \cite{naboko18}, the authors employed  a geometric method to construct embedded eigenvalues. While they are able to construct finitely many eigenvalues, to embed infinitely many eigenvalues they require a rational independence condition which our result does not require. }

Note also that the proof that the construction produces the desired set of eigenvalues is more difficult in the Jacobi setting compared to the continuous Schr\"odinger setting. The spectral transition of embedded eigenvalues for discrete operators heavily depends on the arithmetic properties of quasimomenta. For example, the sharp transition for a single embedded eigenvalue for the continuous was known 40 years ago \cite{atk}, dating back to \cite{Hal}. However, similar results for the discrete case  are still open \cite{liu2018criteria}. In addition, the generalized Pr\"ufer  transformations are singular for the discrete setting.  Although  the proof of the continuous and discrete case looks similar, the understanding and mathematical principles behind them are significantly different.  In this paper,
 the construction for the continuous case  can be bounded by a constant in the continuous case, but in the Jacobi setting those same terms are bounded by a term that grows like $\varepsilon \ln n$ for small positive $\varepsilon$ and as $n\to\infty$, which  leads to an additional parameter in the construction.\color{black}

Our paper is organized in the following way. In Section \ref{Section:Main} we will introduce notation and state our results.  We first address our proofs in the continuous Schr\"odinger setting.
In Section \ref{Section:Small} we will prove a result complementary to our main results: that no embedded eigenvalues will be produced if our perturbation is small. \dco{Section \ref{Section:Summary} is when we begin to address our main theorem. This section is just a non-technical summary of our method, aimed to give the reader an intuition about how our construction works.} We will prove important technical estimates in Section \ref{Section:Prep}, and in Section \ref{Section:Construct} we will show how to construct $V(x)$.
In the next sections, we prove results in the Jacobi setting. In Section \ref{sec:Jacobi1}, we discuss Pr\"ufer variables and the discrete analogue of our auxiliary small perturbation result. In Section \ref{sec:Jacobi2} we prove our main results concerning embedded eigenvalues, mainly explaining the parts of the proof that differ from the continuous Schr\"odinger setting. For the readers' convenience, we write out explicitly the proofs for the Jacobi setting in the Appendix.
\section{Main Results}\label{Section:Main}
We consider a Floquet solution $\varphi$ of \eqref{GV}, which has the following form
\begin{equation}\label{Gfqp}
    \varphi(x,E)=p(x,E)e^{i k(E)x}
\end{equation}
 where $k(E)$ is the quasimomentum, and $p(x,E)$ is 1-periodic.

 It is known that the  spectrum of $H_0$ (on the whole line)  is purely absolutely continuous and
consists of a union   closed
intervals (often referred to as   bands). We denote
\begin{equation*}
  \sigma_{\rm ac}(H_0) = \sigma_{\rm ess}(H_0)=\bigcup_{k}[c_k,d_k].
\end{equation*}
In each band $[c_k,d_k]$, $k(E)$  is monotonically increasing from $0$ to $\pi$ or monotonically decreasing from $\pi$ to $0$.
Any two of those bands  can intersect at most at one  point.
By Weyl's theorem, $\sigma_{\rm ess}(H)=\sigma_{\rm ess}(H_0)$ if $\limsup_{x\to \infty}|V(x)|=0$.
\begin{theorem}\label{mainthm1}
Suppose
\begin{equation}\label{Gassum1}
    V(x)=\frac{o(1)}{1+x}
\end{equation}
as $x\to \infty$.  Let $H=H_0+V$.  Then there exists no non-trivial $L^2(\R^+)$ solution of $Hu=Eu$ for any $E\in \cup_k (c_k,d_k)$. More precisely,
if for some  $E\in\cup _n(c_k,d_k)$  the solution $u$ of $Hu=Eu$ satisfies $u\in L^2(\R^+)$, then $u\equiv 0$.

\end{theorem}
\begin{theorem}\label{mainthm2}
Suppose $\{ E_j\}_{j=1}^N\subset \cup_k (c_k,d_k)$  such that quasimomenta $\{k(E_j)\}_{j=1}^N$ are different.
Suppose for any $i,j\in\{1,2,\cdots,N\} $, $k(E_i)+k(E_j)\neq \pi$.
Then for any given $\{\theta_j\}_{j=1}^N\subset [0,\pi]$,
there exist  functions $V\in C^{\infty}[0,\infty)$ such that
\begin{equation}\label{Ggoalb}
    V(x)=\frac{O(1)}{1+x}
\end{equation}
as $x\to \infty$ and
\begin{equation*}
   Hu=E_ju
\end{equation*}
has an $L^2(\R^+)$ solution with boundary condition
\begin{equation*}
    \frac{u^\prime(0)}{u(0)}=\tan\theta_j.
\end{equation*}

\end{theorem}
\begin{corollary}\label{cor1}
Choose any band $[c_k,d_k]$. Let $e_k\in [c_k,d_k]$  be such that $k(e_k)=\frac{\pi}{2}$.
Suppose $\{ E_j\}_{j=1}^N$ are a finite set of distinct points in $(c_k,e_k)$ or $(e_k,d_k)$.
Then for any given $\{\theta_j\}_{j=1}^N\subset [0,\pi]$,
there exist  functions $V\in C^{\infty}[0,\infty)$ such that
\eqref{Ggoalb} holds as $x\to \infty$ and
\begin{equation*}
   Hu=E_ju
\end{equation*}
has an $L^2(\R^+)$ solution with boundary condition
\begin{equation*}
    \frac{u^\prime(0)}{u(0)}=\tan\theta_j.
\end{equation*}

\end{corollary}

\begin{theorem}\label{mainthm3}
Suppose $A=\{ E_j\}_{j=1}^\infty\subset \cup_n (a_n,b_n)$  such that quasimomenta $\{k(E_j)\}_{j}$ are different.
Suppose for any $i,j $, $k(E_i)+k(E_j)\neq \pi$.
Let $h(x)>0$ be   any
function  on  $(0,\infty)$  with $  \lim_{x\to \infty}h(x) = \infty$.

Then for any given $\{\theta_j\}_{j=1}^\infty\subset [0,\pi]$, there exist  functions $V\in C^{\infty}[0,\infty]$ such that
\begin{equation}\label{Ggoala}
    |V(x)|\leq \frac{h(x)}{1+x}\quad \text{for } x>0,
\end{equation}
  and
\begin{equation*}
   Hu=E_ju
\end{equation*}
has an $L^2(\R^+)$ solution with boundary condition
\begin{equation*}
    \frac{u^\prime(0)}{u(0)}=\tan\theta_j.
\end{equation*}
\end{theorem}

\begin{corollary}\label{coro3}
Choose any band $[c_k,d_k]$. Let $e_k\in [c_k,d_k]$  be such that $k(e_k)=\frac{\pi}{2}$. Suppose $\{ E_j\}_{j=1}^{\infty}$ are a countable set of distinct points in $(c_k,e_k)$ or $(e_k,d_k)$.
Let $h(x)>0$ be   any
function  on  $(0,\infty)$  with $  \lim_{x\to \infty}h(x) = \infty$.

Then for any given $\{\theta_j\}_{j=1}^\infty\subset [0,\pi]$, there exist  functions $V\in C^{\infty}[0,\infty]$ such that \eqref{Ggoala} holds
  and
\begin{equation*}
   Hu=E_ju
\end{equation*}
has an $L^2(\R^+)$ solution with boundary condition
\begin{equation*}
    \frac{u^\prime(0)}{u(0)}=\tan\theta_j.
\end{equation*}

\end{corollary}

\begin{remark}\ \\
\begin{enumerate}[(i)]
  \item Actually, in the  proof of Theorems \ref{mainthm2} and \ref{mainthm3}, we show that
\begin{equation*}
    V^{(k)}(x)=\frac{O_k(1)}{1+x}
\end{equation*}
and
\begin{equation*}
    |V^{(k)}(x)|\leq O_k(1)\frac{h(x)}{1+x}, \quad x>0
\end{equation*}
respectively, where $O_k(1)$ is a large constant depending on $k$.
  \item  Although we  only consider the half line $[0,\infty)$, all the results in this paper hold for $x\in(-\infty,0]$.
  \item  We can assume   $V(x)$ we constructed  in Theorems  \ref{mainthm2}, \ref{mainthm3} and Corollaries  \ref{cor1}, \ref{coro3} satisfies
  \begin{equation*}
    |V(x)|\leq \frac{C}{(1+|x|)^{\frac{2}{3}}}.
  \end{equation*}
  Thus $\sigma_{ac}(H)=\sigma_{ac}(H_0)=\cup_k [c_k,d_k]$ \cite{acmk}.
\end{enumerate}

\end{remark}
Now we are in the position to introduce the results for perturbed periodic Jacobi operators.
 Recalling the equation \eqref{Jeq:JacobiUnperturbed} we denote
\begin{equation*}
  \sigma_{\rm ac}(J_0) = \sigma_{\rm ess}(J_0)=\bigcup_{k}[c_k,d_k].
\end{equation*}
Let $E\in (c_k,d_k)$ and $\varphi$  be the Floquet solution of $q$-periodic operator.
Suppose
\begin{equation}\label{JGflo}
 \varphi(n,E)= p(n) e^{i  \frac{k(E)}{q}n},
 \end{equation}
 where $p(n)$ is a real $q$-periodic function and $k(E)\in(0,\pi)$ is called the quasimomentum ($q$ is the period for $a_n,b_n$).
 Sometimes, we omit the dependence on $E$.

  \begin{theorem}\label{Jthm:NoEmbedde}

Suppose $a_n^{\prime}=\frac{o(1)}{1+n} $ and  $b_n^{\prime}=\frac{o(1)}{1+n} $.
 Let $J$ be given by \eqref{Jeq:JacobiPerturbed}.  Then there exists no non-trivial $\ell^2(\Z_{\geq 0})$ solution of $Ju=Eu$ for any $E\in \cup_k (c_k,d_k)$.
 \end{theorem}
  \begin{theorem}\label{Jthm:FiniteEmbedded}
Suppose $\{ E_j\}_{j=1}^N\subset \cup_k (c_k,d_k)$  such that quasimomenta $\{k(E_j)\}_{j=1}^N$ are different.
Suppose for any $i,j\in\{1,2,\cdots,N\} $, $k(E_i)+k(E_j)\neq \pi$. Let $a_n^{\prime}=0$.
Then for any given $\{\theta_j\}_{j=1}^N\subset [0,\pi]$,
there exist   $b_n^{\prime}$ such that
\begin{equation}\label{JGgoalb}
    b_n^{\prime}=\frac{O(1)}{1+n}
\end{equation}
as $n\to \infty$ and the
\begin{equation*}
   Ju=E_ju
\end{equation*}
has an $\ell^2(\Z_{\geq 0})$ solution with boundary condition
\begin{equation*}
    \frac{u(1)}{u(0)}=\tan\theta_j.
\end{equation*}

 \end{theorem}
 \begin{theorem}\label{Jthm:InfiniteEmbedded}
 Suppose $\{ E_j\}_{j=1}^\infty\subset \cup_k (c_k,d_k)$  such that quasimomenta $\{k(E_j)\}_{j}$ are different.
Suppose for any $i,j $, $k(E_i)+k(E_j)\neq \pi$.
Let $h(n)>0$ be   any
function  on  $\Z_{\geq 0}$  with $  \lim_{n\to \infty}h(n) = \infty$. Let $a_n^\prime=0$

Then for any given $\{\theta_j\}_{j=1}^\infty\subset [0,\pi]$, there exist  sequence $b_n^\prime$ such that
\begin{equation}\label{JGgoala}
    |b^\prime(n)|\leq \frac{h(n)}{1+n}\quad \text{for } n,
\end{equation}
  and
\begin{equation*}
   Ju=E_ju
\end{equation*}
has an $\ell^2(\Z_{\geq 0})$ solution with boundary condition
\begin{equation*}
    \frac{u(1)}{u(0)}=\tan\theta_j.
\end{equation*}
 \end{theorem}
\liu{Finally, we remark that it is possible to make $O(1)$ in  \eqref{Ggoalb} and  \eqref{JGgoalb} quantitative \cite{liustark18}.
Also, under the assumption $V(x)=\frac{O(1)}{1+|x|}$, we can show that the singular continuous spectrum of $H_0+V$ is empty \cite{liusc2}. Similar results hold for the discrete cases \cite{liusc1}.}
\section{Absence of embedded eigenvalues for small perturbations in the continuous setting}\label{Section:Small}
From Section \ref{Section:Small} through Section \ref{Section:Construct}, we only consider continuous Schr\"odinger operators.

Let $E\in \cup_n(a_n,b_n)$ and let $\varphi(x,E)$ be the Floquet solution of $H_0$.
We   recall the generalized Pr\"ufer transformation of Schr\"odinger equation $Hu=Eu$ first, which is from \cite{KRS}.

 By interchanging $\varphi$ and $\overline{\varphi}$, we can assume
 \begin{equation*}
    {\rm Im}(\overline{\varphi(0)}\varphi^\prime(0))>0.
 \end{equation*}

Define  $\gamma(x,E)$  as a continuous function such that
\begin{equation}\label{Gfl}
  \varphi(x,E) =|\varphi(x,E)|e^{i \gamma(x,E)}.
\end{equation}
In the following arguments, we leave the dependence on $E$ implicit if there is no confusion. Note that we define $u$ to be a real solution of \eqref{Gu} and $\varphi$ is a Floquet solution of \eqref{GV} (so $\varphi$ is complex-valued). We also assume the  quasimomentum $k(E)$ satisfies $0\leq k( E)\leq \pi$.

By \cite[Proposition 2.1]{KRS}, we know there exists some constant  $G >0$ (depending on $E$) such that
\begin{equation}\label{Ggammad}
   \frac{1}{G}\leq \gamma^\prime(x)\leq G.
\end{equation}

\begin{proposition}[ Proposition 2.2 and Theorem 2.3bc of \cite{KRS}] \label{thmformula}
Suppose $u$ is a real solution of \eqref{Gu}.
Then there exist real functions $ R(x)>0$ and $\theta(x)$ such that
\begin{equation}\label{GformulaR}
    [\ln R(x)]^\prime=\frac{V(x)}{2\gamma^\prime(x)} \sin2 \theta(x)
\end{equation}
and
\begin{equation}\label{Gformulatheta}
   \theta(x)^\prime=\gamma^\prime(x)-\frac{V(x)}{2\gamma^\prime(x)} \sin^2 \theta(x).
\end{equation}

Moreover,  there exists a constant $K$(depending on $E$) such that
\begin{equation}\label{Gformulanorm}
    \frac{|u(x)|^2+|u^\prime(x)|^2}{K} \leq R(x)^2\leq K (|u(x)|^2+|u^\prime(x)|^2).
\end{equation}

\end{proposition}
\begin{remark}
Let $\delta(x)$ be continuous function such that
\begin{equation}\label{Gformulare1}
    \varphi^{\prime}(x)=i|\varphi^{\prime}(x)| e^{i\delta(x)}.
\end{equation}
The we have the following precise relations,
\begin{equation}\label{Gformulare2}
    u(x)=R(x)|\varphi(x)|\sin\theta(x)
\end{equation}
and
\begin{equation}\label{Gformulare3}
    u^{\prime}(x)=R(x)|\varphi^\prime(x)|\cos(\theta(x)+\delta(x)-\gamma(x)).
\end{equation}
\end{remark}

\begin{proof}[Proof of Theorem \ref{mainthm1}]
Suppose $u$ is an eigensolution with corresponding eigenvalue $E\in (a_n,b_n)$.
By   \eqref{GformulaR} and the assumption \eqref{Gassum1}, we have
\begin{equation}
    \ln R(x)\geq \ln R(x_0) -\frac{1}{3}\int_{x_0}^{x} \frac{1}{1+x}dx
\end{equation}

for large $x_0$ and $x>x_0$.  Fixing $x_0$, we obtain for large $x$ and a constant $\widetilde C$,
\begin{equation*}
    R(x)\geq \frac{1}{\widetilde C x^{\frac{1}{3}}}.
\end{equation*}
This contradicts \eqref{Gformulanorm} and $u\in L^2(\R^+)$. Here we used the basic fact that $Hu=Eu$ and $u\in L^2(\R^+)$ imply  $u^{\prime}\in L^2(\R^+)$.
\end{proof}
\section{\dco{A non-technical summary of our method}}\label{Section:Summary}

\dco{Since the calculations in the next few sections will be very technical and complicated, let us first provide a non-technical summary of our technique, to help the reader understand how everything connects in the big picture. The challenge of our construction is that we are trying to create many diffferent eigenvalues (perhaps a countably infinite number) simulateneously. In other words, our solution must decay fast enough for many different values of the energy $E$; let's say we desire a $V(x)$ that induces embedded eigenvalues at $E=E_1, E_2, E_3 \ldots$. The difficulty is, if we create a potential $V_{E_1}(x)$ that produces a decaying eigensolution $u_{E_1}$ that corresponds to an energy $E_1$, that potential might cause solutions $u_{E_2}(x), u_{E_3}(x), \ldots$ corresponding to $E_2,E_3,\ldots$ to grow.}

\dco{
We thus perform a complicated concatentation process on the potential $V(x)$ to ensure that all the eigensolutions $u_{E_1}(x), u_{E_2}(x), u_{E_3}(x),\ldots$ decay quickly. At each stage of the concatenation (think of a stage as an interval in $[0,\infty)$ ), we construct a potential that forces the eigensolution corresponding to a single energy to decay. In one stage, we construct $V_{E_1}(x)$ so that the eigensolution $u_{E_1}$ corresponding to an energy $E_1$ decays very quickily, while we prove upper bounds on how much the eigensolutions $u_{E_2}, u_{E_3}, \ldots$ corresponding to the other desired energies can grow. Then we concatenate a next stage $V_{E_2}(x)$, that makes eigenfunction corresponding to a second energy $E_2$ decay quickly, while limiting how much the other eigenfunctions $u_{E_1}(x), u_{E_3}(x), \ldots$ can grow, et cetera. We then alternate these stages. If we have a finite number of $E_1, E_2, E_3, \ldots E_k$ we simply repeat the $V_{E_1}(x), V_{E_2}(x),\ldots, V_{E_k}(x),$ stages periodically. If we have infinitely many $E_1, E_2, E_3,\ldots$ the concatenation gets more complicated,  but it is still possible to alternate the stages in such a way that the $V_{E_k}(x)$ concatenation occurs infinitely many times for every $k$ (albeit each $V_k(x)$ concatenation occurs more and more rarely as $x$ increases).  We construct the $V_{E_1}(x), V_{E_2}(x), \ldots$ in such a way that each eigensolution decays quickly enough at the stages where we are focusing on them, that it compensates for how they might grow when we are focusing on other eigensolutions. We perform this delicate procedure and this results in all the desired eigensolutions decaying quickly enough to be in $\ell^2$.}

\dco{
Intuitively, our construction works in the following way: we bound the eigenfunctions by integrals that involve decaying oscillatory terms, for instance involving sines and cosines. It is unsurprising that we can do this, since in our setting the background potential is periodic and the perturbative potential we construct is formed by concatenating chopped-up pieces of decaying oscillatory functions. We then carefully show that for the decaying oscillatory terms in our integral bound, the positive parts of the oscillation mostly cancels out with the negative part, and this results in small upper bounds for the sizes of our eigenfunctions.
}

\dco{In Section \ref{Section:Prep} we will prove various lemmas that show integrals of various oscillating expressions are small. This will culminate in Proposition \ref{Twocase}, which is the proposition that asserts that there exists a $V_{E_1}(x)$ that ensures that the eigenfunction $u_{E_1}(x)$ decays very quickly, and the other eigenfunctions $u_{E_2}(x),u_{E_3}(x),\ldots$ will not grow too much. In Section \ref{Section:Construct} we explain how we concatenate the $V_{E_1}(x), V_{E_2}(x),\ldots $ stages, and we prove that we do indeed get eigenvalues where we desire them }

\dco{For the discrete case, although the calculations are different the idea is more or less the same as what we explained above for the continuous case.}

\section{Some preparations for construction in the continuous setting}\label{Section:Prep}
Before we proceed with our perturbative construction, we will have to lay some groundwork to ensure that certain key terms decay quickly enough for our purposes. This section is the most novel and difficult of our paper, and demonstrates clearest why perturbing a periodic operator is more challenging than perturbing a free operator.

 For any $E\in(a_n,b_n)$, we consider the non-linear differential equation for $x>b$,
 \begin{equation}\label{Gnonlinear}
    \theta^\prime(x,E,a,b,\theta_0)=\gamma^\prime(x,E)+\frac{C}{\gamma^\prime(x,E) (1+x-b)}\sin2\theta\sin^2\theta,
 \end{equation}
 where $C$ is a large constant that will be chosen later. Solving  \eqref{Gnonlinear} on $[a,\infty)$ with
 initial condition $\frac{ \theta^{\prime}(a)}{\theta(a)}=\tan \theta_0$, where $a>b$, we get a unique solution.
 Notice that $\theta$ depends on $a,\theta_0$ and $E$.
 Set
 \begin{equation}\label{GnonlinearV}
     {V}(x,E,a,b,\theta_0)=-\frac{C}{1+x-b}\sin2\theta(x).
 \end{equation}

 \begin{proposition}\label{Fourier}
 Suppose $\theta(x,E,a,b,\theta_0)$ is given by  \eqref{Gnonlinear}, $k(E)\neq \frac{\pi}{2}$ and  ${V}(x,E,a,b,\theta_0)$ is given by \eqref{GnonlinearV}.  Then we have
   \begin{equation}\label{Gfourier1}
    \int_{a}^x \frac{1}{1+y}\cos4\theta(y) dy=O(1).
\end{equation}
Let $\hat{E}$ be another energy  in  $\cup_{\ell}(a_{\ell},b_{\ell})$ such that $k(\hat{E})\neq k(E)$ and $k(\hat{E})+ k(E)\neq \pi$.
Suppose  $\theta(x,\hat{E})$ is a solution of
\begin{equation*}
   \theta^\prime(x,\hat{E})=\gamma^\prime(x,\hat{E})-\frac{{V}(x,E,a,b,\theta_0)}{2\gamma^\prime(x,\hat{E})} \sin^2 \theta(x,\hat{E}).
\end{equation*}
Then
 \begin{equation}\label{Gfourier2}
    \int_{x_0}^x\frac{1}{ 2\gamma^{\prime}(y,\hat{E})}\frac{1}{1+y-b}\sin2\theta(y,{E})\sin2\theta(y,\hat{E}) dy=O(\frac{1}{x_0-b}),
\end{equation}
for any $x>x_0>a$.
 \end{proposition}
 \begin{proof}
 We only give the proof of \eqref{Gfourier2}.  The proof of \eqref{Gfourier1} is similar. Without loss of generality, we  assume $x_0>a$ is large.
 First, using \eqref{Gformulatheta} and \eqref{Gnonlinear} we have the differential equations of $\theta(x,E)$ and $\theta(x,\hat{E})$,
\begin{equation}\label{Gfequ1}
\theta^\prime(x,E)=\gamma^\prime(x,E)+\frac{C}{\gamma^\prime(x,E) (1+x-b)}\sin2\theta(x,E)\sin^2\theta(x,E),
\end{equation}
and
\begin{equation}\label{Gfequ2}
\theta^\prime(x,\hat{E})=\gamma^\prime(x,\hat{E})+\frac{C}{\gamma^\prime(x,\hat{E}) (1+x-b)}\sin2\theta(x,E)\sin^2\theta(x,\hat{E}).
\end{equation}
By \eqref{Gfqp} and \eqref{Gfl}, we have
\begin{equation}\label{Gfequ3}
    \gamma (x,E)=k(E)x+\eta(x,E),
\end{equation}
where $\eta(x,E)\mod 2\pi$ is a function that is $1$-periodic in $x$.

Observe that by basic trigonometry,

\begin{equation}
-2\sin 2\theta(y,{E})\sin 2\theta(y,\hat{E})=\cos (2\theta(y,{E})+2\theta(y,\hat{E}))-\cos (2\theta(y,{E})-2\theta(y,\hat{E})).
\end{equation}

Thus it suffices for us to find a bound for

\begin{equation}\label{Dec6}
\int_{x_0}^x\frac{\cos (2\theta(y,{E})\pm2\theta(y,\hat{E}))}{ 2\gamma^{\prime}(y,\hat{E})(1+y-b)} dy.
\end{equation}
For simplicity, let us focus on the $2\theta(y,{E})-2\theta(y,\hat{E})$ case. The $2\theta(y,{E})+2\theta(y,\hat{E})$ case will proceed  in a similar way.

By \eqref{Gfequ1},\eqref{Gfequ2} and \eqref{Gfequ3}, we have
\begin{equation}\label{GfourierTheta}
\dfrac{d}{d x}([\theta(x,{E})-\eta(x,E)]-[\theta(x,{\hat E})-\eta(x,\hat{E})]=(k(E)-k(\hat {E}))+\frac{O(1)}{1+x-b}.
\end{equation}
Let
\begin{equation*}
  \tilde{ {\theta}}(x,E) =\theta(x,{E})-\eta(x,E),
\end{equation*}
and
\begin{equation*}
     \tilde{\theta}(x,\hat{E}) =\theta(x,\hat{E})-\eta(x,\hat{E}).
\end{equation*}
 By trigonometry again,
 one has
 \begin{eqnarray*}
   \cos (2\theta(x,{E})-2\theta(x,\hat{E})) &=& \cos (2\tilde{\theta}(x,{E})-2\tilde{\theta}(x,\hat{E})+2\eta(x,E)-2\eta(x,\hat{E})) \\
    &=& \cos(2\eta(x,E)-2\eta(x,\hat{E}))\cos (2\tilde{\theta}(x,{E})-2\tilde{\theta}(x,\hat{E}))\\
    &&-\sin(2\eta(x,E)-2\eta(x,\hat{E}))\sin (2\tilde{\theta}(x,{E})-2\tilde{\theta}(x,\hat{E})).
 \end{eqnarray*}
 Thus
\begin{eqnarray*}
  \int_{x_0}^x\frac{\cos (2\theta(y,{E})-2\theta(y,\hat{E}))}{ 2\gamma^{\prime}(y,E)(1+y-b)} dy &=& \int_{x_0}^x\frac{\cos(2\eta(y,E)-2\eta(y,\hat{E}))}{ 2\gamma^{\prime}(y,E)}\frac{\cos (2\tilde{\theta}(y,{E})-2\tilde{\theta}(y,\hat{E}))}{1+y-b} dy \\
 &&- \int_{x_0}^x\frac{\sin(2\eta(y,E)-2\eta(y,\hat{E}))}{ 2\gamma^{\prime}(y,E)}\frac{\sin (2\tilde{\theta}(y,{E})-2\tilde{\theta}(y,\hat{E}))}{1+y-b} dy.
\end{eqnarray*}
%
%
Again, because the estimate of the other term follows in a similar way, we only give the estimate for
\begin{equation}\label{Gfourier5}
\int_{x_0}^x\frac{\sin(2\eta(y,E)-2\eta(y,\hat{E}))}{ 2\gamma^{\prime}(y,E)}\frac{\sin (2\tilde{\theta}(y,{E})-2\tilde{\theta}(y,\hat{E}))}{1+y-b} dy.
\end{equation}


We proceed by Fourier expansion of $\frac{\sin(2\eta(x,E)-2\eta(x,\hat{E}))}{\gamma^{\prime}(x,E)}$ (which is $1$-periodic continuous) and obtain that
\begin{equation*}
\frac{\sin(2\eta(x,E)-2\eta(x,\hat{E}))}{\gamma^{\prime}(x,E)}=\frac{c_0}{2}+\sum_{k=1}^\infty c_k \cos(2\pi kx)+ d_k \sin(2\pi kx).
\end{equation*}

Plugging this back into \eqref{Gfourier5}, we get

\begin{align}\nonumber \eqref{Gfourier5}=\int_{x_0}^x \frac{c_0}{2} \frac{\sin (2\tilde{\theta}(y,{E})-2\tilde{\theta}(y,\hat{E}))}{(1+y-b)}dx+\sum_{k=1}^\infty c_k \cos(2\pi ky)\frac{\sin (2\tilde{\theta}(y,{E})-2\tilde{\theta}(y,\hat{E}))}{(1+y-b)}dy\\+\sum_{k=1}^\infty d_k \sin(2\pi ky)\frac{\sin (2\tilde{\theta}(y,{E})-2\tilde{\theta}(y,\hat{E}))}{(1+y-b)}dy.\label{eq:Fouriersum}
\end{align}
By  the Cauchy-Schwarz inequality, \eqref{eq:Fouriersum} and the fact that
    $\sum c_\ell^2+d_\ell^2<\infty$, we only need to show that for $\ell>0$
\begin{equation*}
  \int_{x_0}^x  \cos(2\pi \ell y)\frac{\sin (2\tilde{\theta}(y,{E})-2\tilde{\theta}(y,\hat{E}))}{(1+y-b)}dy=\frac{1}{\ell}O\left(\frac{1}{x_0-b}\right)
\end{equation*}
and
\begin{equation}\label{Gfourier6}
  \int_{x_0}^x  \sin(2\pi \ell y)\frac{\sin (2\tilde{\theta}(y,{E})-2\tilde{\theta}(y,\hat{E}))}{(1+y-b)}dy=\frac{1}{\ell}O\left(\frac{1}{x_0-b}\right),
\end{equation}
and
\begin{equation*}
    \int_{x_0}^x   \frac{\sin (2\tilde{\theta}(y,{E})-2\tilde{\theta}(y,\hat{E}))}{(1+y-b)}dx=O\left(\frac{1}{x_0-b}\right).
\end{equation*}

As before, we only give the proof of \eqref{Gfourier6}.

By trigonometry, we have
\begin{align}\nonumber
  \int_{x_0}^x  \sin(2\pi \ell y)\frac{\sin (2\tilde{\theta}(y,{E})-2\tilde{\theta}(y,\hat{E}))}{(1+y-b)}dy=&  \int_{x_0}^x \frac{\cos(2\pi \ell y- (2\tilde{\theta}(y,{E})-2\tilde{\theta}(y,\hat{E}))) }{2(1+y-b)}dy\\
  &-\frac{\cos(2\pi \ell y+ (2\tilde{\theta}(y,{E})-2\tilde{\theta}(y,\hat{E}))) }{2(1+y-b)}dy.
  \label{Gfourier7}
\end{align}
By the same reason, we only prove that
\begin{equation}\label{Gfourier8}
    \int_{x_0}^x \frac{\cos(2\pi \ell y- (2\tilde{\theta}(y,{E})-2\tilde{\theta}(y,\hat{E}))) }{1+y-b}dy=\frac{1}{\ell}\frac{O(1)}{x_0-b{+1}}.
\end{equation}

Since   $k(E)$ and $ k(\hat E)$ are distinct, we must have 
\begin{equation}
0<|k(E)- k(\hat E)|< \pi.
\end{equation}
Note that since the other case has a minus instead of a plus, here is where we need the restriction $k(E)+k(\hat E)\neq\pi$.

Denote
\begin{equation*}
    \tilde{\theta}_\ell(x)=2\pi \ell x- 2(\tilde{\theta}(x,{E})-\tilde{\theta}(x,\hat{E})),
\end{equation*}
and
\begin{equation*}
    \tilde{\ell}=2\pi \ell-2(k(E)-k(\hat{E}))>0.
\end{equation*}

By \eqref{GfourierTheta}, one has
\begin{equation}\label{Gfourier9}
    \tilde{\theta}'_\ell(x)=\tilde{\ell}+\frac{O(1)}{1+x-b}
\end{equation}
Observe that this is positive if $x-b$ is sufficiently large.

 Let $i_0$ be the largest   integer such that $2\pi i_0+\frac{\pi}{2}<\tilde{\theta}_\ell(x_0)$.
 By \eqref{Gfourier9},
 there exist
  $x_0<x_1<x_2<\cdots<x_t<x_{t+1}$  such that $x$ lies in $[x_{t-1},x_t)$ and
 \begin{equation}\label{tildetheta}
   \tilde{\theta}_\ell(x_i)= 2\pi i_0+\frac{2i-1}{2}\pi
 \end{equation}
 for $i=1,2,\cdots,t,t+1$.

 By integrating \eqref{Gfourier9}, we obtain
 \begin{equation}
  \tilde\theta_\ell(x)=\tilde{\ell}x+O(1)\ln(1+x-b)
 \end{equation}
 And so
 \begin{align*}
 \tilde\ell|x_{i+1}-x_{i}|=&\tilde\theta_\ell(x_{i+1})-\tilde\theta_\ell(x_{i})+O(1)\ln\left(\frac{1+x_{i+1}-b}{1+x_i-b}\right)\\
 =&\pi+O(1)\ln\left(\frac{1+x_{i+1}-b}{1+x_i-b}\right)\text{by \eqref{tildetheta}}\\
  = &\pi+O(1)\left|\ln\left(1+\frac{x_{i+1}-x_i}{1+x_i-b}\right)\right|\\
   = &\pi+O(1)\left|\frac{x_{i+1}-x_i}{1+x_i-b}\right|.
 \end{align*}
 This implies
 \begin{equation*}
    |x_{i+1}-x_{i}|=\frac{\pi}{ \tilde{\ell}}+ \frac{O(1)}{\tilde \ell(x_i+1-b)},
 \end{equation*}

and so for sufficiently large $x_i-b$,
 \begin{equation}\label{e2}
  x_i\geq x_0+\frac{i\pi}{2\tilde{\ell}}.
 \end{equation}
 Similarly, for $y\in[x_i,x_{i+1})$, we have

 \begin{eqnarray*}
    \tilde{\theta}_\ell(y) 
     &=& 2\pi i_0+i \pi-\frac{\pi}{2}+\tilde{\ell}(y-x_i)+\frac{O(1)}{\tilde \ell (1+x_i-b)}.
 \end{eqnarray*}
 Which implies
 \begin{eqnarray}
   \nonumber && \int_{x_i}^{x_{i+1}}|\cos(2\pi \ell y- 2(\tilde{\theta}(y,{E})-\tilde{\theta}(y,\hat{E}))) |dy\\
   \nonumber  &=&\int_{x_i}^{x_{i+1}}|\cos\tilde{\theta}_\ell(y)|dy \\
     &=&2\int_{0}^{\frac{\pi}{2\tilde{\ell}}}\cos( \tilde{\ell}y )dy+ \frac{O(1)}{\tilde \ell ^2(1+x_i-b)} =\frac{1}{\tilde{\ell}}+\frac{O(1)}{\tilde \ell^2(1+x_i-b)}. \label{e1}
 \end{eqnarray}
Notice that $\cos(2\pi \ell x- 2(\tilde{\theta}(x,{E})-\tilde{\theta}(x,\hat{E})))$ changes the sign at $x_i$. The integral also has some cancellation between $(x_{i-1},x_i)$ and $(x_{i },x_{i+1})$.
Let $t^\prime\in\{t,t+1\}$ such that ${t}^\prime $ is odd.

By  \eqref{e1}, we obtain

 \begin{eqnarray}
  \nonumber\int_{x_0}^x\frac{\cos(2\pi \ell y-2 (\tilde{\theta}(y,{E})-\tilde{\theta}(y,\hat{E})))}{1+y-b}dy&=&\frac{O(1)}{\tilde\ell(1+x_0-b)}+
   \int_{x_1}^{x_{t^\prime}}\frac{\cos\tilde{\theta}_\ell(y)}{1+y-b}dy \\
  \nonumber   &=& \frac{O(1)}{\tilde\ell(1+x_0-b)}+\sum_{i=1}^{t+1}\frac{O(1)}{\tilde\ell^2(1+ x_i-b)}\frac{1}{1+x_i-b}\\
      &=&\frac{O(1)}{\tilde\ell(1+x_0-b)},\label{ellestimate}
 \end{eqnarray}
where the last equality holds by \eqref{e2}. Since $\ell/\tilde\ell$ is bounded, \eqref{Gfourier8} follows. This concludes our proof.
\end{proof}
\begin{remark}
In order to estimate the other part of  \eqref{Dec6}, that is
\begin{equation*}
\int_{x_0}^x\frac{\cos (2\theta(y,{E})+2\theta(y,\hat{E}))}{ 2\gamma^{\prime}(y,\hat{E})(1+x-b)} dy,
\end{equation*}
we need the assumption $k(E)+k(\hat{E})\neq \pi$.
\end{remark}
 \begin{lemma}\label{thm2}
 Fix $E\in (a_n,b_n)$ and boundary condition $\theta_0\in[0,\pi)$. Then there exists a $\psi_0\in [0,\pi)$ such that
 under the potential of $V$ given by \eqref{GnonlinearV}, the  solution of $Hu=Eu$ on $[a,\infty)$ with boundary condition $\frac{u^{\prime}(a)}{u(a)}=\tan\theta_0$ satisfies
 \begin{equation}\label{Gdecay}
  \ln  R(x,E) -\ln R(a,E)\leq -100\ln \frac{x-b}{a-b}+C
 \end{equation}
  and
 \begin{equation}\label{Gdecay1}
  \ln  R(x,E) \leq \ln R(a,E)
 \end{equation}
 for all $x>a$.
 \end{lemma}
 \begin{proof}
 Without loss of generality, assume $b=0$.
 Choose some $\psi_0=\theta(a)$ such that
\eqref{Gfl}, \eqref{Gformulare1}, \eqref{Gformulare2} and \eqref{Gformulare3} hold for $x=a$ and  $\frac{u^{\prime}(a)}{u(a)}=\tan\theta_0$.
By \eqref{GformulaR}, \eqref{Gformulatheta}, \eqref{Gnonlinear} and \eqref{GnonlinearV}, we have
\begin{equation}\label{Gdecayformula}
    \ln  R(x,E) -\ln R(a,E)=-\int_{a}^x\frac{ C}{ 2\gamma^{\prime}(y,E)}\frac{1}{1+x}\sin^22\theta(y) dy
\end{equation}
and
\begin{equation}
\theta^\prime(x,E)=\gamma^\prime(x,E)+\frac{C}{2\gamma^\prime(x,E) (1+x)}\sin2\theta\sin^2\theta.
\end{equation}
Observe that \eqref{Gdecay1} follows from \eqref{Gdecayformula} directly.

By \eqref{Gfourier1} in Proposition \ref{Fourier}, one has
\begin{equation*}
    \int_{a}^x \frac{1}{1+y}\cos4\theta(y) dy=O(1).
\end{equation*}
This yields that
\begin{eqnarray*}
  -\int_{a}^x\frac{C}{ 2\gamma^{\prime}(y,E)}\frac{1}{1+y}\sin^22\theta(y) dy &= &  -\int_{x_0}^x\frac{C}{ 4\gamma^{\prime}(y,E)}\frac{1}{1+y}(1-\cos4\theta(y)) dy \\
   &\leq & - \int_{a}^x\frac{C}{ 4\gamma^{\prime}(y,E)}\frac{1}{1+y} dy\\
   &\leq &  -100\ln \frac{x}{a}+C.
\end{eqnarray*}
 \end{proof}

 \begin{lemma}\label{lemmax2}
 Let us use the potential $V(x,E,a,b)$ of Lemma \ref{thm2} in \eqref{Gu}. Let $\hat{E}$ be another energy  in  $\cup_{\ell}(a_{\ell},b_{\ell})$ such that $k(\hat{E})\neq k(E)$ and $k(\hat{E})+ k(E)\neq \pi$.
 Then we have
 \begin{equation}\label{Gnotchange}
    R(x,\hat{E})\leq  1.5 R(x_0,\hat{E}),
 \end{equation}
 for any $x>x_0\geq a$ and large enough $x_0-b$.
 \end{lemma}
 \begin{proof}

 By \eqref{GformulaR} and \eqref{GnonlinearV}, we have
\begin{equation*}
    \ln  R(x,\hat{E}) -\ln R(x_0,\hat{E})=-\int_{x_0}^x\frac{ C}{ 2\gamma^{\prime}(y,\hat{E})}\frac{1}{1+x}\sin2\theta(y,{E})\sin2\theta(y,\hat{E}) dy.
\end{equation*}

By \eqref{Gfourier2} in the previous Proposition \ref{Fourier},
\begin{equation*}
    \int_{x_0}^x\frac{1}{ 2\gamma^{\prime}(y,\hat{E})}\frac{1}{1+y-b}\sin2\theta(y,{E})\sin2\theta(y,\hat{E}) dy=O(\frac{1}{x_0-b}),
\end{equation*}
for all $x>x_0\geq a$.
This implies Lemma \ref{lemmax2}.
 \end{proof}

So far we have a construction of $V$ that is discontinuous. Now we want to assert that we may choose $V$ to be smooth.
\begin{proposition}\label{Twocase}
Let  $E$ and $ A=\{\hat{E}_j\}_{j=1}^k$ be in $\cup_{\ell}(a_{\ell},b_{\ell})$.  Suppose $k(E)$ and $\{k(\hat{E}_j)\}_{j=1}^k$ are different, and
$k(E)+k(\hat{E}_j)\neq \pi$ for any $j\in\{1,2,\cdots,N\}$.
Suppose  $\theta_0\in[0,\pi]$. Let $x_1>x_0>b$.
Then there exist constants $K(E, A)$, $C(E, A)$ (independent of $b, x_0$ and $x_1$) and potential $\widetilde V(x,E,A,x_0,x_1,b,\theta_0)$  such that  for $x_0-b>K(E,A)$ the following holds:

   \begin{description}
     \item[Potential]   for $x_0\leq x \leq x_1$, ${\rm supp}(\widetilde V)\subset(x_0,x_1)$,  $\widetilde V\in C^{\infty}(x_0,x_1)$,  and
     \begin{equation}\label{thm141}
        |\widetilde V(x,E,A,x_0,x_1,b,\theta_0)|\leq \frac{C(E, A)}{x-b}
     \end{equation}

     \item[Solution for $E$] the solution of $(H_0+\widetilde V)u=Eu$ with boundary condition $\frac{u^\prime(x_0)}{u(x_0)}=\tan\theta_0$ satisfies
     \begin{equation}\label{thm142}
        R(x_1,E)\leq C(E, A)(\frac{x_1-b}{x_0-b})^{-100} R(x_0,E)
     \end{equation}
     and  for $x_0<x<x_1$,
      \begin{equation}\label{thm143}
        R(x,E)\leq  2 R(x_0,E).
     \end{equation}
      \item[Solution for $\hat{E}_j$] the solution of $(H_0+\widetilde V)u=\hat{E}_ju$ with any boundary condition satisfies
      for $x_0<x\leq x_1$,
      \begin{equation}\label{thm144}
        R(x,\hat{E}_j)\leq  2 R(x_0,\hat{E}_j).
     \end{equation}
   \end{description}

  \end{proposition}
  \begin{proof}
  Let $V_1$ be given by \eqref{GnonlinearV} with $ a=x_0$.
  Let $x=x_1$ and $a=x_0$ in Lemmas \ref{thm2} and \ref{lemmax2}.
  We modify $V_1$ on boundary $x=x_0$ and $x=x_1$ a little and obtain $V$. We can also require $|V(x)|\leq |V_1(x)|$. Recall that $R$ is the magnitude of the solution of the linear  differential equation \eqref{GV}. Thus $R(x,E)$ is continuously related to $V$, and so a small  change in $V$ will only result in a small change in $R(x,E)$ in the finite interval  $x\in [x_0,x_1]$. Thus Lemmas \ref{thm2} and \ref{lemmax2} still hold, and this implies Proposition \ref{Twocase}.

%
%
  \end{proof}
 \section{Constructing the perturbative potential in the continuous setting}\label{Section:Construct}
In this section we will give a proof of Theorems \ref{mainthm2} and \ref{mainthm3}. We will give the construction of the potential $V$. The idea is to glue the potential $V(x,E,A,x_0,x_1,b,\theta_0)$ in a piecewise manner. Our construction is inspired by \cite{jitomirskaya2018noncompact}, where they use it to construct a rotationally symmetric metric on manifolds.

 Let us fix a band of the absolutely continuous spectrum, and  enumerate the desired embedded eigenvalues in our band spectrum as $E_j$ (we always assume there are countably many). Let $N:\mathbb Z^+\to \mathbb Z^+$ be a non-decreasing  function, $N(1)=1$ and $N(w)$ grows  very slowly  (in other words, we expect $N(w)=N(w+1)$ to be true for ``most" $w\in\mathbb Z_+$). Furthermore, we define $N$ so if $N(w+1)>N(w)$ then $N(w+1)=N(w)+1$.
Let $C_{w}$ be a large constant that depends on the eigenvalues $E_1$ until  $E_{N(w)}$
\begin{equation}\label{DeCk}
 C_{w}=C(E_1,E_2,
\cdots,E_{N(w)}).
\end{equation}
We emphasize that the dependence of $C_{w+1}$ on the $E_j$ does not take into account multiplicity. Thus if $N(w+1)=N(w+2)$ (which we expect to happen very frequently) then $C_{w+1}=C_{w+2}$.

We have $N(w)=\max N$ for sufficiently large $w$ in the construction of Theorem \ref{mainthm2} and we instead have $\lim_{w}N(w)=\infty$ in the construction of Theorem \ref{mainthm3}.

 Define
 \begin{equation}\label{Twdef}
 T_{w+1}=T_wC_{w+1}
 \end{equation} and $T_0=C_1$.
By modifying $C_w$, we can assume   $T_w$ is large enough so that  \[T_w\geq K(E,\{E_j\}_{j=1}^{N(w)}\backslash E)\] for any $E\in \{E_j\}_{j=1}^{N(w)}$  in   Proposition \ref{Twocase}.

On the other hand,  if $N(w)  $  goes to infinity arbitrarily slowly, then  $C_w  $   can also  go  to infinity arbitrarily slowly. This doesn't contradict our previous statement that $T_w$ is ``large enough", since we can choose the $C_w$ to be large but also choose it to be constant for long stretches of $w\in\mathbb Z_+$.
We do however choose $C_w$ so that it goes to infinity faster than $N(w)$: let us in fact choose $C_w$ so that
\begin{equation}\label{addliu}
    C_{w}\geq 4^{N(w+1)}.
\end{equation}

We can also assume
\begin{equation}\label{GTk}
    T_w\geq 1000^{w}.
\end{equation}
and for large $w$,
\begin{equation*}
    C_w\leq \ln w,
\end{equation*}
and
\begin{equation}\label{Gapr92}
  C_w^2 N(w)\leq \frac{1}{100} \min_{x\in [J_{w-1},J_w]}h(x),
\end{equation}
where $h(x)$ is given by Theorem \ref{mainthm3}.

Let
\begin{equation}\label{Jwdef}
J_w=\sum_{i}^wN(i) T_i .
\end{equation} Notice that $ J_w$ and $T_w$ go to infinity faster than $C_w$. More precisely, we will have $C_w/J_w$ and $C_w/T_w$ both tending to $0$ as $w$ tends to infinity.

We will also define function $V$ ($\text{supp} V\subset (1,\infty)$ ) and $u(x,E_j)$, $j=1,2,\ldots $ on $(1,J_w)$ by induction, such that
\begin{enumerate}[1.]
\item
$u(x,E_j)$ solves  for $x\in (0,J_w)$
\begin{align}\label{eigenengj}
     \left( -\frac{d^2}{dx^2} + V_0(x)+V(x)\right) u(x,E_j)
     =E_j  u(x,E_j),
\end{align}
and satisfies  boundary condition
\begin{equation}\label{1boundaryn}
    \frac{u^{\prime}(0,E_j)}{u(0,E_j)}=\tan\theta_j,
\end{equation}
\item
$u(x,E_i)$  for $i=1,2,\cdots, N(w)$ and $w\geq 2$, satisfies
\begin{eqnarray}
R(J_w,E_i)
  &\leq& 2^{N(w)}  N(w)^{50}C_{w}^{-50}R(J_{w-1},E_i).\label{eigenj}
\end{eqnarray}

\item  $V(x)\in C^{\infty}(J_{w-1},J_w]$  and
\begin{equation}\label{controlkr}
    | V(x) |\leq M\frac{N(w)C^2_{w}}{x+1},
\end{equation}
where $M$ is an absolute constant.
\end{enumerate}

By our construction, one has
\begin{eqnarray}
  \frac{J_w}{T_{w+1}} &\leq & 2 \frac{\sum_{i}^wN(i)T_i}{T_{w+1}} \\
  &\leq & 2\frac{N(w) }{C_{w+1}}\sum_{i=1}^w\frac{T_i}{T_w}\\
  &\leq & 4\frac{N(w) }{C_{w+1}}\label{Tk1Jk}.
\end{eqnarray}

The last inequality comes from \eqref{GTk}. 

\subsection{Construction}

Define $V(x)=0$ for $x\in[0,1]$. 
 Let  $u(x,E_j)$ be the solution of
\begin{equation}\label{definu}
   H u=E_ju
\end{equation}
with boundary  condition
\begin{equation*}
    \frac{u^{\prime}(0,E_j)}{u(0,E_j)}=\tan\theta_j.
\end{equation*}

 We proceed by an induction argument. Suppose  we completed the construction  $V(x)$ for step $w$. That is we have given the definition of $u(x,E_j)$ on $(1,J_w]$ for all possible $j$.
Suppose also  $u(x,E_i)$ on $(1,J_w]$ for $i=1,2,\cdots,N(w)$ satisfies \eqref{eigenj}.

Denote $B_{w+1}=\{E_i\}_{i=1}^{N(w+1)}$.
Applying Proposition \ref{Twocase} to $x_0=J_w$, $x_1=J_w+T_{w+1}$, $b=0$, $E=E_1$, $\tan\theta_0=\frac{u^\prime(J_w,E_1)}{u(J_w,E_1)}$ and $A= B_{w+1}\backslash \{E_1\}$,  we can define
$V(x,E_1,B_{w+1}\backslash \{E_1\},J_w,J_{w}+T_{w+1},0,\theta_0)$  on $x\in (J_w, J_w+T_{w+1}]$ since the boundary condition matches at the point $ J_w$ (guaranteed by $\tan\theta_0=\frac{u^\prime(J_w,E_1)}{u(J_w,E_1)}$). Thus we can define $u(x,E_j) $ on $(0,J_{w}+T_{w+1})$  for all possible $j$.
 Moreover,  letting $x_1=J_w+T_{w+1}$  in Proposition \ref{Twocase}, one has (by \eqref{thm142})
 \begin{eqnarray}
    R(J_w+T_{w+1},E_1) &\leq& (\frac{J_w+T_{w+1}}{J_w})^{-100}C_{w+1} R(J_w,E_1) \nonumber\\
    &\leq&   N(w)^{50} C_{w+1}^{-50} R(J_w,E_1), \label{Gkstep1}
 \end{eqnarray}
 since \eqref{Tk1Jk} holds and $C_{w+1}$  is chosen to be large.

We mention that now   the constant $C(E,A)$ in  Proposition \ref{Twocase} should be $C_{w+1}$.

 Applying  Proposition \ref{Twocase}  to $x_0=J_w+T_{w+1}$, $x_1=J_w+2T_{w+1}$, $b=T_{w+1}$, $E=E_2$, $A=B_{w+1}\backslash E_2$, and $\tan \theta_0=\frac{u^\prime(J_w+T_{w+1},E_2)}{u(J_w+T_{w+1},E_2)}$,  we can define
 ${V}(x,E_2, B_{w+1}\backslash E_2, J_w+T_{w+1}, J_w+2T_{w+1},T_{w+1},\theta_0)$ on $x\in (J_w+T_{w+1}, {J_w}+2T_{w+1}]$. Thus we can define $u(x,E_j) $ on $(0,J_{w}+2T_{w+1})$   for all possible $j$.
 Moreover,  letting $x_1=J_w+2T_{w+1}$ in  Proposition \ref{Twocase}, one has
 \begin{eqnarray}
   R(J_w+2T_{w+1},E_2) &\leq& (\frac{J_w+T_{w+1}}{J_w})^{-100}C_{w+1} R(J_w+T_{w+1},E_2)  \nonumber\\
     &\leq&  N(w)^{50} C_{w+1}^{-50} R(J_{w}+T_{w+1},E_2).\label{Gkstep2}
 \end{eqnarray}

 Suppose we give the definition  of $V$ and $u(x,E_j)$ for all $j$ on $(0,J_w+tT_{w+1}]$ for $t\leq N(w+1)-1$. Let us give the definition on $(0,J_w+(t+1)T_{w+1}]$.

 Applying  Proposition \ref{Twocase}  to $x_0=J_w+tT_{w+1}$, $x_1=J_w+(t+1)T_{w+1}$, $b=tT_{w+1}$, $E=E_{t+1}$, $A=B_{w+1}\backslash E_{t+1}$ and $\tan \theta_0=\frac{u^\prime(J_w+tT_{w+1},E_{t+1})}{u(J_w+tT_{w+1},E_{t+1})}$,  we can define
 ${V}(x,E_{t+1}, B_{w+1}\backslash E_{t+1}, J_w+tT_{w+1},J_w+(t+1)T_{w+1},tT_{w+1},\theta_0)$ on $x\in (J_w+tT_{w+1}, J_w+(t+1)T_{w+1})$. Thus we can define $u(x,E_j) $ on $(0,J_{w}+(t+1)T_{w+1}]$  for all possible $j$.
 Moreover,  letting $x_1=J_w+(t+1)T_{w+1}$ in  Proposition \ref{Twocase}, one has
 \begin{eqnarray}
    R(J_w+(t+1)T_{w+1},E_{t+1}) &\leq& (\frac{J_w+T_{w+1}}{J_w})^{-100}C_{w} R(J_w+tT_{w+1},E_{t+1}) \nonumber\\
   &\leq&  N(w)^{50} C_{w+1}^{-50} R(J_{w}+tT_{w+1},E_{t+1}). \label{Gkstept}
 \end{eqnarray}
 Thus we can define  on $(0,J_w+N(k+1)T_{w+1})=(0,J_{w+1})$ by induction for $ J_w+tT_{w+1}$.

 Let us mention that for $x\in[J_w+tT_{w+1},J_w+(t+1)T_{w+1}]$ and $0\leq t\leq N(w+1)-1$,
 \begin{equation}\label{Consv}
    V(x)= \widetilde V\left(x,E_{t+1}, B_{w+1}\backslash \{E_{t+1}\},J_k+tT_{k+1},J_k+(t+1)T_{k+1}, tT_{k+1},\frac{u^\prime(J_w+tT_{w+1},E_{t+1})}{u(J_w+tT_{w+1},E_{t+1})}\right),
 \end{equation}
 where $\widetilde V$ is taken from Proposition \ref{Twocase}.

 Now we should show that the definition satisfies the $w+1$ step conditions \eqref{eigenengj}-\eqref{controlkr}.

 Let us consider $R(x,E_i)$ for $i=1,2,\cdots,N(w+1)$.
 $R(x,E_i)$ decreases  from   point  $J_w+(i-1)T_{w+1}$ to $J_w+iT_{w+1}$, $i=1,2,\cdots,N(w+1)$, and
  may increase from any point $J_w+(m-1)T_{w+1}$ to $J_w+mT_{w+1}$, $m=1,2,\cdots,N(w+1)$ and $m\neq i$.
  That is
  \begin{equation*}
    R(J_w+iT_{w+1},E_i)\leq N^{50}(w)C_{w+1}^{-50} R(J_w+(i-1)T_{w+1},E_{i}),
  \end{equation*}
  and   for $m\neq i$,
  \begin{equation*}
    R(J_w+mT_{w+1},E_i)\leq 2 R(J_w+(m-1)T_{w+1},E_{i}),
  \end{equation*}
  by  Proposition \ref{Twocase}.

   Thus for $i=1,2,\cdots,N(w+1)$,
   \begin{equation*}
    R(J_{w+1},E_i)\leq 2^{N(w+1)} N(w)^{50} C_{w+1}^{-50} R(J_{w},E_i).
   \end{equation*}

This implies (\ref{eigenj}) for $w+1$.

 By the construction of $V(x)$ \eqref{Consv}, \eqref{thm141}, and \eqref{Twdef} we have
 for $x\in[J_w+tT_{w+1},J_w+(t+1)T_{w+1}]$ and $0\leq t\leq N(w+1)-1$,
 \begin{equation}\label{b'1}
   |V(x)| \leq \frac{C_{w+1}}{x-tT_{w+1}}
   \leq \frac{C_{w+1}}{(J_w+tT_{w+1})-tT_{w+1}}
   = \frac{C_{w+1}}{J_w}.
 \end{equation}

 Furthermore, notice that by \eqref{Twdef} and \eqref{Jwdef}, for a constant $M$,

 \begin{equation}
\frac{T_{w+1}}{J_w}=\frac{T_{w}C_{w+1}}{J_w}<M C_{w+1}
 \end{equation}

 Recall that  $x\in[J_w+tT_{w+1},J_w+(t+1)T_{w+1}]$ and $0\leq t\leq N(w+1)-1$.

Direct computations show that
 \begin{align}
 \frac{1}{N(w+1)}+\frac{1}{J_wN(w+1)}+\frac{T_{w+1}}{J_w}<& 2C_{w+1}\nonumber\\
 1+\frac{1}{J_w}+N(w+1)\frac{T_{w+1}}{J_w}<& 2N(w+1)C_{w+1}\nonumber\\
  \frac{J_w+N(w+1)T_{w+1}+1}{J_w}<& 2N(w+1)C_{w+1}\nonumber\\
 \frac{1}{J_w}<& \frac{2N(w+1)C_{w+1}}{J_w+N(w+1)T_{w+1}+1}\nonumber\\
  \frac{1}{J_w}<& \frac{2N(w+1)C_{w+1}}{J_w+(t+1)T_{w+1}+1}\nonumber\\
    \frac{1}{J_w}<& \frac{100N(w+1)C_{w+1}}{x+1}\nonumber\\
    \frac{C_{w+1}}{J_w}<& \frac{100N(w+1)C_{w+1}^2}{x+1}\label{b'2}
 \end{align}

 By \eqref{b'1} and \eqref{b'2} we have for $x\in[J_w,J_{w+1}]$,
\begin{equation}
|V(x)| <   100\frac{N(w+1)C^2_{w+1}}{x+1}.\label{b'bound}
\end{equation}

 This implies \eqref{controlkr}.


\subsection{Proof of Theorems \ref{mainthm2} and \ref{mainthm3}}
\begin{proof}
In the construction of Theorem \ref{mainthm2}, eventually $N(w)$ and $C_w$ are bounded.
In the construction of Theorem \ref{mainthm3},   $N(w)$ and $C_w$ grow to infinity arbitrarily slowly.
By \eqref{controlkr} and \eqref{Gapr92}, \eqref{Ggoalb} and \eqref{Ggoala} hold.

By (\ref{Gformulanorm}), it suffices to show that for any $j$, $R(x,E_j) \in L^2([1,\infty),dx)$.
Below we give the details.

For any $N(w_0-1)<j\leq N(w_0)$, by the construction (see (\ref{eigenj})), we have for $w\geq w_0$
\begin{eqnarray}
    R(J_{w+1},E_j)&\leq& 2^{N(w+1)}  N(w)^{50} C_{w+1}^{-50} R(J_{w},E_j) \nonumber\\
    &\leq&    C_{w+1}^{-25} R(J_{w},E_j) \nonumber \\
    &\leq& T_{w_0}^{25}T_{w+1}^{-25} R(J_{w_0},E_j)\label{lastiteration}
\end{eqnarray}
where the second  inequality holds by \eqref{addliu} and the third inequality holds by \eqref{Twdef}.

By (\ref{thm143}), (\ref{thm144}), \eqref{addliu} and \eqref{lastiteration}, for all $x\in[J_{w+1},J_{w+2}]$,
\begin{eqnarray}
  R(x,E_j) & \leq&   2^{N(w+2)}R(J_{w+1},E_j)\nonumber\\
  &\leq&   C_{w+2}R(J_{w+1},E_j)\nonumber\\
     &\leq&
     T_{w_0}^{25}T_{w+1}^{-24} R(J_{w_0},E_j).\label{lastGx}
\end{eqnarray}
\

Then by  (\ref{lastGx}), we have
\begin{eqnarray*}
  \int_{J_{w_0+1}}^{\infty} R^2(x,E_j) dx&=&\sum_{w\geq w_0} \int_{J_{w+1}}^{J_{w+2}} R^2(x,E_j)dx\\
   &\leq&   \sum_{w\geq w_0}\int_{J_{w+1}}^{J_{w+2}}  T_{w_0}^{50}T_{w+1}^{-48} R^2(J_{w_0},E_j) dx\\
   &\leq& T_{w_0}^{50}R^2(J_{w_0},E_j) \sum_{w\geq w_0} N(w+2)T_{w+2}T_{w+1}^{-48} \\
   &=&  T_{w_0}^{50} R^2(J_{w_0},E_j)\sum_{w\geq w_0}N(w+2)C_{w+2}T_{w+1}^{- 47} \\
     &\leq&  T_{w_0}^{50} R^2(J_{w_0},E_j)\sum_{w\geq w_0}T_{w+1}^{-40} <\infty,
\end{eqnarray*}
since $N(w)$ and $C_w$ go to infinity slowly and $T_w$ satisfies \eqref{GTk}.
This completes the proof.
\end{proof}

\section{Generalized Pr\"ufer transformation and proof of Theorem \ref{Jthm:NoEmbedde}}\label{sec:Jacobi1}
This section is mostly a summary of the generalized Pr\"ufer variables developed in \cite{lukdcmv}. At the end of the section, we prove Theorem \ref{Jthm:NoEmbedde}. In \eqref{Jeq:JacobiUnperturbed}, we have a Jacobi matrix $  J$ with coefficients $a_n > 0$, $b_n \in \mathbb{R}$,  viewed as an operator $J_0$ on $\ell^2(\mathbb Z_{\geq 0})$. We consider also its perturbation, a Jacobi matrix $  { J}$ with coefficients $a_n + a_n' > 0$, $b_n + b_n' \in \mathbb{R}$, and  viewed as an operator $J$ on $\ell^2(\mathbb Z_{\geq 0})$. For  $E\in \cup (c_k,d_k)$, let $\varphi$ be a Floquet solution given by \eqref{JGflo}.\liu{ Without loss of generality, assume $|\varphi(0)|^2+|\varphi(1)|^2=1$.}
Obviously,
\begin{equation}\label{Jvarphi}
a_{n+1}\varphi(n+1)+b_{n+1}\varphi(n)+a_n \varphi(n-1) = E\varphi(n),
\end{equation}
We also consider an eigensolution $u$ for $H$,
\begin{equation}\label{Ju}
(a_{n+1}+ a_{n+1}')u(n+1)+(b_{n+1}+ b_{n+1}')u(n)+(a_n+ a_n') u(n-1) = Eu(n).
\end{equation}
We define $\gamma(n)$ as the argument of $\varphi(n)$. In other words,
 \begin{equation}\label{JGflo1}
 \varphi(n)=\vert\varphi(n)\vert e^{i\gamma(n)}.
 \end{equation}
We can ensure uniqueness of $\gamma$ by setting $\gamma(0)\in [0,2\pi)$, $\gamma(n)-\gamma(n-1)\in [0,2\pi)$.

Note that $\varphi$ is complex, and is linearly independent with its complex conjugate $\bar\varphi$. On the other hand, we assume that $u$ is a real-valued eigensolution.

We now introduce $Z(n)$. Our Pr\"ufer variables will be define as the argument and absolute value of $Z(n)$. It is defined as follows:
\begin{align}
\label{Jrho1}
\begin{pmatrix}
(a_n+ a_n')u(n)\\
u(n-1)
\end{pmatrix}=&\frac{1}{2i}\left(
Z(n)\begin{pmatrix}
a_n\varphi(n)\\
\varphi(n-1)
\end{pmatrix}
-\overline {Z(n)}
\begin{pmatrix}
a_n\overline{\varphi(n)}\\
\overline{\varphi(n-1)}
\end{pmatrix}
\right)
\\
=&
\label{Jrho2}\mathrm{Im}\left[
Z(n)
\begin{pmatrix}
a_n\varphi(n)\\
\varphi(n-1)
\end{pmatrix}
\right].
\end{align}
By linear independence of $\varphi$ and $\bar\varphi$ and reality of $u$, \eqref{Jrho1} uniquely determines $Z(n)$. The Pr\"ufer amplitude $R(n) > 0$ and Pr\"ufer phase $\eta(n) \in \mathbb{R}$ are defined as
\begin{equation}\label{JZReta}
Z(n) = R(n) e^{i\eta(n)}.
\end{equation}

We will also need a few alternate versions of the Wronskian. For two sequences $f,g$, we have
\begin{align*}
W_{0,0}(f,g)(n)=& a_{n+1} f(n)g(n+1)-a_{n+1}f(n+1)g(n),\\
W_{a',a'}(f,g)(n)=& (a_{n+1}+a_{n+1}') f(n)g(n+1)-(a_{n+1}+a_{n+1}')f(n+1)g(n),\\
W_{0,a'}(f,g)(n)=& (a_{n+1}+a_{n+1}') f(n)g(n+1)-a_{n+1}f(n+1)g(n).
\end{align*}

If we assume
    \[a_{n+1}f(n+1)+a_nf(n-1)=(x-b_{n+1})f(n),\] and \[(a_{n+1}+a_{n+1}')g(n+1)+(a_{n}+a_{n}')g(n-1)=(x-b_{n+1}-b_{n+1}')g(n),\]
   then
 \begin{align}\nonumber
 W_{0,a'}(f,g)(n)-W_{0,a'}(f,g)(n-1)=&-b'_{n+1}f(n)g(n)\\
 &-a'_n(f(n)g(n-1)+f(n-1)g(n)).\label{JWronskianDifference}
 \end{align}
 Since $\varphi, \overline\varphi$ are linearly independent solutions of (\ref{Jvarphi}), by constancy of the Wronskian, we have
 \begin{equation}\label{Jomegadefn}
 W_{0,0}(\overline\varphi,\varphi)(n)=2ia_{n+1}\mathrm{Im}(\overline{\varphi(n)}\varphi(n+1))=i\omega,
 \end{equation}
 for some real nonzero constant $\omega$.
Thus,
 \begin{equation}\label{Jgamma-omega}
 2 \vert\varphi(n)\vert\cdot\vert \varphi(n+1)\vert a_{n+1}\sin(\gamma(n+1)-\gamma(n))=\omega.
\end{equation}
We can use Wronskians to invert (\ref{Jrho1}) to get
\begin{equation}\label{Jrho-omega}
Z(n)=\frac{2}{\omega}W_{0,a'}(\overline \varphi,u)(n-1).
\end{equation}

\begin{theorem}[Theorem 5 of \cite{lukdcmv}]\label{Jt.OPRL} Pr\"ufer variables obey the first-order recursion relation
\begin{align*} &\frac{Z(n+1)}{Z(n)} \\
 = &  1-\frac{i}{\omega} \frac{a_n}{a_n+a_n'} b_{n+1}'\vert \varphi(n)\vert^2 (e^{-2i(\eta(n)+\gamma(n))} -1)\\
&+\frac{i}{\omega}  a_{n}'\vert \varphi(n-1)\vert\cdot \vert \varphi(n)\vert e^{i(\gamma(n-1)-\gamma(n))}\\
&-\frac{i}{\omega}  a_{n}'\vert \varphi(n-1)\vert\cdot \vert \varphi(n)\vert e^{-2i\eta(n)}e^{-i(\gamma(n-1)+\gamma(n))}\\
&+\frac{i}{\omega}\frac{a_n}{a_n+a_n'} a_{n}' (1-e^{-2i(\eta(n)+\gamma(n))})\vert \varphi(n-1)\vert\cdot \vert \varphi(n)\vert e^{-i(\gamma(n-1)-\gamma(n))}.
\end{align*}
\end{theorem}
\begin{remark}
\liu{In this paper, we assume  $a_n^{\prime}= o(1) $ and $b_n^{\prime}= o(1) $. } Since
$a_n,b_n$ are periodic, $a_n,a_n+a_n^{\prime}>0$ for all $n$, then
\begin{equation}\label{JGbound}
  \frac{1}{a_n+a_n^{\prime}}=O(1).
\end{equation}
\end{remark}
We define the Pr\"ufer amplitude $R$ and the Pr\"ufer phase $\eta$ by
\begin{equation}
R(n)=\vert Z(n)\vert, \eta(n)=\mathrm{Arg}(Z(n)).
\end{equation}
 In that case, we have

 \begin{equation}\label{Jeq:R/R=|Z/Z|}
 \frac{R(n+1)}{R(n)}=\left \vert \frac{Z(n+1)}{Z(n)}\right\vert.
  \end{equation}

Note the following bound on $R(n)$:

\begin{proposition} \label{JPboundUR}For a constant $K$ (depending on $H_0$ and $E$),
\begin{equation*}
\liu{ \frac{1}{K}\sqrt{u(n)^2+u(n-1)^2} \leq   R(n)\leq K\sqrt{u(n)^2+u(n-1)^2}.}
\end{equation*}

%
\end{proposition}
\begin{proof}
The left inequality simply follows from  \eqref{Jrho2}. The right inequality follows from \eqref{Jrho-omega} and the Cauchy-Schwarz inequality.
\end{proof}

Let us set $a'_j=0$ for all $j$. This changes Theorem \ref{Jt.OPRL} into a much simpler formula,

\begin{equation}\label{Jeq:SimplerFormula}
\frac{Z(n+1)}{Z(n)} \\
 =   1-\frac{i}{\omega}  b_{n+1}'\vert \varphi(n)\vert^2 (e^{-2i(\eta(n)+\gamma(n))} -1)\\
\end{equation}

 Using \eqref{Jeq:SimplerFormula} and \eqref{Jeq:R/R=|Z/Z|} we have

 \begin{align}
 \frac{R(n+1)^2}{R(n)^2}=& \left(1-\frac{i}{\omega}b_{n+1}' \vert\varphi(n)\vert^2(\cos(2\eta(n)+2\gamma(n))-i\sin(2\eta(n)+2\gamma(n))-1)\right)\nonumber\\
 &\times\left(1+\frac{i}{\omega}b_{n+1}' \vert\varphi(n)\vert^2(\cos(2\eta(n)+2\gamma(n))+i\sin(2\eta(n)+2\gamma(n))-1)\right)\nonumber\\
 =& 1-b_{n+1}'\frac{2}{\omega}\sin(2\eta(n)+2\gamma(n))\vert \varphi(n)\vert^2\nonumber \\
 &+\frac{4(b'_{n+1})^2\vert \varphi(n)\vert^4}{\omega^2}\sin^2(\eta(n)+\gamma(n)).\label{eq:sin2}
 \end{align}

Also, starting with \eqref{Jeq:SimplerFormula} and multiplying by $Z(n)e^{i\gamma(n)}$ we obtain

\begin{align}\nonumber
&R(n+1)\exp(i\eta(n+1)+i\gamma(n))\\
=&R(n)\exp(i\eta(n)+i\gamma(n))\nonumber -\frac{i}{\omega}b'_{n+1}\vert \varphi(n)\vert^2(\exp(-i\eta(n)\\
&-i\gamma(n))-\exp(i\eta(n)+i\gamma(n)))R(n).
\end{align}
Dividing the real part by the imaginary part for both sides of the above equation, we get

\begin{equation}\label{Jeq:cot}
\cot(\eta(n+1)+\gamma(n))=\cot(\eta(n)+\gamma(n)) -\frac{2}{\omega}b'_{n+1}\vert \varphi(n)\vert^2
\end{equation}

\begin{proof}[ \bf Proof of Theorem \ref{Jthm:NoEmbedde}]
Suppose $u$ is an eigensolution with corresponding   $E\in (c_k,d_k)$.
By Theorem \ref{Jt.OPRL}, \eqref{JGbound} and  \eqref{Jeq:R/R=|Z/Z|}, we have
\begin{equation*}
  \frac{R(n+1)}{R(n)}=1-\frac{o(1)}{n}.
\end{equation*}
This implies that
\begin{equation}
  \ln R(n+1)-\ln R(n)=\frac{o(1)}{n}.
\end{equation}
Thus for large $n_0$, and $n>n_0$, we have
\begin{equation}
  \ln R(n)\geq \ln R(n_0)-\frac{1}{3}\sum_{k=n_0}^n\frac{1}{k}.
\end{equation}
This implies for large $n$,
\begin{equation*}
  R(n)\geq  \frac{1}{Cn^{\frac{1}{3}}}.
\end{equation*}
This contradicts $u\in \ell^2(\Z_{\geq 0})$ by Proposition \ref{JPboundUR}.
\end{proof}

 \section{The perturbative construction in the Jacobi setting}\label{sec:Jacobi2}
We always assume  $a_n^\prime=0$. In this section, all the equations are in the discrete setting.  We indicate the dependence on $E$; thus we will write $R(n,E)$, $Z(n,E)$, $\eta(n,E)$  and $\gamma(n,E)$.
Let  $\theta (n,E)=\eta(n,E)+\gamma(n,E)$.

By \eqref{Jeq:cot} and \cite[Prop.2.4]{kiselev1998modified}, one has
\begin{equation*}
    (\eta(n+1)+\gamma(n))- (\eta(n)+\gamma(n))=O( |b'_{n+1}|).
\end{equation*}
This implies
\begin{equation}\label{JLT}
  \theta(n+1,E)-\theta(n,E)=\gamma(n+1,E)-\gamma(n,E)+O(|b_{n+1}^\prime|).
\end{equation}
We will add  another  equation to complete our construction. Using \eqref{eq:sin2} we get:
\begin{equation}
   \ln R(n+1,E) -\ln R(n,E)= - \frac{b_{n+1}'}{\omega}\sin(2\eta(n,E)+2\gamma(n,E))\vert \varphi(n,E)\vert^2 + O(|b_{n+1}^\prime|^2) .\label{JLR}
\end{equation}

We will construct $b_n^\prime$ in a piecewise manner.  Let $J_0$ be the periodic operator with Jacobi coefficient sequences $a_n,b_n$ and $J_0+b'\text{Id}$ be the perturbation with coefficient sequences $a_n, b_n+b'_n$.
\begin{proposition}\label{JTwocase}
Let  $E $ be in $\cup_\ell(c_\ell,d_\ell)$  such that $k(E)\neq \frac{\pi}{2}$. Let $ A=\{{E}_j\}_{j=1}^m$ be in $\cup_\ell(c_\ell,d_\ell)$ such  that $k(E)\neq k(  {E}_j)$ and $k(E)+k( {E}_j)\neq \pi$ for all $j=1,2,\cdots,m.$
Suppose  $\theta_0\in(0,\pi)$. Let $n_1>n_0>v$.
Then there exist constants $K(E, A)$, $C(E, A)$ (independent of $v, n_0$ and $n_1$) and perturbation $b'_n(E,A,n_0,n_1,v,\theta_0)$  such that  for $n_0-v>K(E,A)$ the following holds:

   \begin{description}
     \item[Perturbation]   for $n_0\leq n \leq n_1$, ${\rm supp}(b')\subset(n_0,n_1)$,  and
     \begin{equation}\label{Jthm141a}
        | b'_n(E,A,n_0,n_1,v,\theta_0)|\leq \frac{C(E, A)}{n-v}.
     \end{equation}

     \item[Solution for $E$] the solution of $(J_0+  b'\mathrm{Id})u=Eu$ with boundary condition $\theta(n_0,E)=\theta_0$ satisfies
     \begin{equation}\label{Jthm142a}
        R(n_1,E)\leq C(E, A)(\frac{n_1-v}{n_0-v})^{-100} R(n_0,E)
     \end{equation}
     and  for $n_0<n<n_1$,
      \begin{equation}\label{Jthm143a}
        R(n,E)\leq C(E,A)  R(n_0,E).
     \end{equation}
     In particular, for any $\varepsilon>0$, if $\frac{n_1-v}{n_0-v}>K(E,A,\varepsilon)$,
     \begin{equation}\label{Jthm1446a}
        R(n,{E})\leq    (\frac{n_1-v}{n_0-v})^{\varepsilon}R(n_0, {E}).
     \end{equation}
      \item[Solution for $ {E}_j$] any  solution of $(H_0+ b'\mathrm{Id})u={E}_ju$ satisfies
      for $n_0<n\leq n_1$ and $\varepsilon>0$,
      \begin{equation}\label{Jthm144a}
        R(n, {E}_j)\leq  D(E,A,\varepsilon) (\frac{n_1-v}{n_0-v})^{\varepsilon}R(n_0, {E}_j).
     \end{equation}
     In particular, if $\frac{n_1-v}{n_0-v}>K(E,A,\varepsilon)$,
     \begin{equation}\label{Jthm1445a}
        R(n,{E}_j)\leq    (\frac{n_1-v}{n_0-v})^{\varepsilon}R(n_0,{E}_j).
     \end{equation}
   \end{description}
 \end{proposition}

For simplicity, denote by $K=K(E,A)$,  $C=C(E,A)$  etc.. We mention that \[K\gg C>0.\]
   Recall that $\gamma(n,E)$ is the argument of $\varphi$ and is therefore fixed.  We solve the following equation for $\eta(n,E)$ with initial condition $\eta(n_0,E)=\theta_0-\gamma(n_0,E)$ (or in other words,  $\theta(n_0,E)=\theta_0$):
   \begin{equation}\label{JPrufTc}
   \cot(\eta(n+1,E)+\gamma(n,E))=\cot(\eta(n,E)+\gamma(n,E)) -\frac{2}{\omega}b'_{n+1}\vert \varphi(n,E)\vert^2
\end{equation}
with
\begin{equation}\label{JGcons2}
   b'_{n+1}=b'_{n+1}(E,A,n_0,n_1,v,\theta_0)=\frac{C}{n-v}\sin(2\eta(n)+2\gamma(n)).
\end{equation}
We will show that this choice of $b'_n$ satisfies our construction.
Obviously,  \eqref{Jthm141a} follows from \eqref{JGcons2}.

First, we require a technical lemma:

\begin{lemma}\label{JLcon1}

Let $b'_n$ be given in \eqref{JGcons2}, and let $E$ and $A$ satisfy the assumptions of Proposition \ref{JTwocase}. Let $f(n)$ be a sequence with $q$ period.
For any $\varepsilon>0$, there exists $D(E,A,\varepsilon)$ such that

\begin{equation}\label{JGcons4}
  \left|  \sum _{t=n_0}^n f(t)\frac{\cos 4 \theta(t,E)}{t-v}\right|\leq D(E,A,\varepsilon)+ {\varepsilon}\ln \frac{n-v}{n_0-v},
\end{equation}
and
\begin{equation}\label{JGcons5}
    \left|\sum_{t=n_0}^n f(t)\frac{\sin 2 \theta(t,E_j) \sin 2 \theta(t,E)}{t-v}\right|\leq D(E,A,\varepsilon)+{\varepsilon}\ln  \frac{n-v}{n_0-v},
\end{equation}
for all $E_j\in A$.
\end{lemma}
\begin{proof}
We only give the proof of \eqref{JGcons4}. The proof of \eqref{JGcons5} proceeds similarly.

Case 1: $\frac{k(E)}{\pi} $ is rational. Since $k(E)\notin \frac{\pi}{2}$,  we can assume  $\frac{k(E)}{\pi}=\frac{N_1}{N}$ for some $N\geq 3$.
Thus for any $\phi$,
 \begin{equation}\label{JGMar21}
    \sum_{j=0}^{N-1} \cos (4   j k(E)+\phi)=0.
 \end{equation}
 By \eqref{JGflo}, \eqref{JGflo1}, \eqref{JLT} and \eqref{Jthm141a}, one has
 \begin{equation}\label{JLT2}
 k(E)= \theta(n_0+q,E)-\theta(n_0,E)+O\left(\frac{1}{n_0-v}\right)\mod \Z.
\end{equation}

 Iterating, we obtain for any positive integer $j\leq N-1$,

 \begin{equation}\label{JLT2b}
 jk(E)= \theta(n_0+jq,E)-\theta(n_0,E)+O\left(\frac{1}{n_0-v}\right)\mod \Z.
\end{equation}

Thus by \eqref{JGMar21} and \eqref{JLT2b}, we can translate $n_0$ by $p$ and use $\phi=\theta(n_0+p,E)$ to get
\begin{equation*}
    \sum_{j=0}^{N-1} \cos 4    \theta(n_0+jq+p,E)=O\left(\frac{1}{n_0+p-v}\right),
 \end{equation*}
 for all $p=0,1,\cdots,q-1$.
 This implies
 \begin{equation*}
    \sum_{j=0}^{N-1}   f(n_0+jq+p) \frac{ \cos 4\theta(n_0+jq+p,E)}{n_0+jq+p}=\frac{O(1)}{(n_0+p-v)^2},
 \end{equation*}
 for all $p=0,1,\cdots,q-1$.
 Let us define an integer $w$ so that $w$ is the largest integer such that $n-n_0\geq Nqw-1$. Then
 \begin{eqnarray}
  \left|  \sum _{t=n_0}^n f(t)\frac{\cos 4 \theta(t,E)}{t-v}\right| &\leq& \frac{|O(1)|}{n_0-v}+\sum_{i=n_0}^{n_0+q-1+Nqw}\frac{|O(1)|}{(i-v)^2} \nonumber \\
    &\leq&\frac{|O(1)|}{n_0-v}+ \sum_{i=n_0}^{\infty}\frac{|O(1)|}{(i-v)^2}\\
    &=&\frac{|O(1)|}{n_0-v}. \label{Jeq:tSum}
 \end{eqnarray}

This completes the proof of \eqref{JGcons4} for rational $\frac{k(E)}{\pi}$.

Case 2: $\frac{k(E)}{\pi}$ is irrational.
By the ergodic theorem,
for any $\varepsilon>0$, there exists $N>0$ such that
 \begin{equation}\label{JGMar22}
    \left|\sum_{j=0}^{N-1} \cos (4   j k(E)+\phi)\right|\leq N\varepsilon.
 \end{equation}
 By \eqref{JGMar22} and \eqref{JLT2}, one has
\begin{equation*}
    \sum_{j=0}^{N-1} \cos 4    \theta(n_0+jq+p,E)\leq N\left(\varepsilon+ O\left(\frac{1}{n_0-v}\right)\right),
 \end{equation*}
 for all $p=0,1,\cdots,q-1$.
 This implies
 \begin{equation}
    \sum_{j=0}^{N-1}  f(n_0+jq+p) \frac{\cos 4  \theta(n_0+jq+p,E)}{n_0+jq-p}\leq N\left (\frac{\varepsilon}{n_0-v}+\frac{O(1)}{(n_0-v)^2}\right),
 \end{equation}
 for all $p=0,1,\cdots,q-1$.

 We note that

 \[
 \sum_{j=n_0}^{n}\frac{1}{j-v}\leq O(1)\ln\left(\frac{n-v}{n_0-v}\right).
 \]
 Thus, performing an estimate analogous to \eqref{Jeq:tSum} we obtain
 \begin{equation*}
     \left|  \sum _{t=n_0}^n f(t)\frac{\cos 4 \theta(t,E)}{t-v}\right|  \leq D(E,A,\varepsilon)+\varepsilon\ln  \left(\frac{n-v}{n_0-v}\right).
 \end{equation*}
This concludes our proof of \eqref{JGcons4} for irrational $\frac{k(E)}{\pi}$.
\end{proof}

   \begin{proof}[\bf Proof of Proposition \ref{JTwocase}]
Equation \eqref{JLR} becomes
\begin{equation}\label{JPrufRcon}
   \ln R(n+1,E)-\ln R(n,E)= -\vert \varphi(n,E)\vert^2 \frac{C}{n-v}\sin^2(2\eta(n)+2\gamma(n))+\frac{\vert O(1)\vert}{(n-v)^2}.
\end{equation}

This implies
\begin{equation}\label{JGcons1a}
    \ln R(n+1,E)-\ln R(n,E) \leq \frac{C}{(n-v)^2}.
\end{equation}
It is easy to see that \eqref{Jthm143a} follows from \eqref{JGcons1a} since $n_0-v>K$.

Rewrite  \eqref{JPrufRcon} as
\begin{equation}\label{JPrufRcon1}
    \ln R(n+1,E)- \ln R(n,E)= -\vert \varphi(n,E)\vert^2 \frac{C}{n-v}+O(1)\vert \varphi(n,E)\vert^2 \frac{\cos 4 \theta(n,E)}{n-v}+\frac{O(1)}{(n-v)^2}.
\end{equation}
Applying \eqref{JGcons4} with $\varepsilon=1$ to \eqref{JPrufRcon1}, we have for $n\geq n_0$,
\begin{eqnarray}
\ln R(n,E)- \ln R(n_0,E) &\leq& \sum_{t=n_0}^n-\frac{C}{t-v}+O(1)\vert \varphi(n,E)\vert^2 \frac{\cos 4 \theta(t,E)}{j-v}+\frac{O(1)}{(t-v)^2}\nonumber\\
   &\leq& C-C\ln (\frac{n-v}{n_0-v}).\nonumber
\end{eqnarray}
This implies \eqref{Jthm142a}.

Now let us consider  the   solution $u(n, {E}_j)$ of $(H_0+b^{\prime}\rm{Id})u={E}_ju$.

By \eqref{JLR} again, one has
\begin{equation*}
  \ln R(n+1,{E}_j) -\ln R(n,{E}_j)=- C \vert \varphi(n,E_j)\vert^2 \frac{\sin 2 \theta(n,E) \sin 2 \theta(n,E_j)}{n-b}+ \frac{O(1)}{(n-b)^2}
\end{equation*}
By  \eqref{JGcons4}   (following Lemma \ref{JLcon1}) and following the proof of \eqref{Jthm142a}, we can prove  \eqref{Jthm144a}. We finish the proof.
   \end{proof}

\begin{proof}[\bf Proof of Theorems   \ref{Jthm:FiniteEmbedded}  and \ref{Jthm:InfiniteEmbedded}]
Replacing Prop.\ref{Twocase} with Prop.\ref{JTwocase}, Theorems   \ref{Jthm:FiniteEmbedded}  and \ref{Jthm:InfiniteEmbedded} can be proved in a similar way of
Theorems \ref{mainthm2} and \ref{mainthm3}. The   difference is that there is a new parameter $\varepsilon$ involved. We write the details in full in the Appendix.
\end{proof}
\appendix
\section{Proof of Theorems   \ref{Jthm:FiniteEmbedded}  and \ref{Jthm:InfiniteEmbedded}}
 We will give the construction of the perturbation $b'$. The idea is to glue the potential $b'(n,E,A,x_0,x_1,v,\theta_0)$ in a piecewise manner like the procedure of continuous case.

 Let us fix a band of the absolutely continuous spectrum, and  enumerate the desired embedded eigenvalues in our band spectrum as $E_j$ (we always assume there are countably many). Let $N:\mathbb Z_{\geq 0}\to \mathbb Z_{\geq 0}$ be a non-decreasing  function, $N(1)=1$ and $N(w)$ grows  very slowly  (in other words, we expect $N(w)=N(w+1)$ to be true for ``most" $w\in\mathbb Z_{\geq 0}$). Furthermore, we define $N$ so if $N(w+1)>N(w)$ then $N(w+1)=N(w)+1$.
 Let
 \begin{equation}\label{JGepsi}
   \varepsilon_w=\frac{1} {100 N(w)}.
 \end{equation}
Let $C_{w}$ be a large constant that depends on the eigenvalues $E_1\ldots E_{N(w)}$. We write
\begin{equation}\label{JDeCk}
 C_{w}=C(E_1,E_2,
\cdots,E_{N(w)}).
\end{equation}
We emphasize that the dependence of $C_{w+1}$ on the $E_j$ does not take into account multiplicity. Thus if $N(w+1)=N(w+2)$ (which we expect to happen very frequently) then $C_{w+1}=C_{w+2}$.
Let
$K_w$ be large enough such that $ K_w>K(E,\{E_j\}_{j=1}^{N(w)}\backslash E,\varepsilon_w)$  for all $E\in \{E_j\}_{j=1}^{N(w)}$ in Proposition \ref{JTwocase}.

We have $N(w)=\max_j N(j)$ for sufficiently large $w$ in the construction of Theorem \ref{Jthm:FiniteEmbedded} and we instead have $\lim_{w}N(w)=\infty$ in the construction of Theorem \ref{Jthm:InfiniteEmbedded}.

 Define
 \begin{equation}\label{JTwdef}
 T_{w+1}=T_wC_{w+1}
 \end{equation} and $T_0=C_1$.
By modifying $C_w$, we can assume   $T_w$ is large enough so that  \[T_w\geq K_w\] for any $E\in \{E_j\}_{j=1}^{N(w)}$  in   Proposition \ref{JTwocase}.

Let $E_j$ and $\theta_j$ be given by Theorem \ref{Jthm:FiniteEmbedded} and Theorem \ref{Jthm:InfiniteEmbedded}.
Fix $w$.
By Proposition \ref{JTwocase},
then there exist constants $K_w$, $C_w$ (independent of $v, n_0$ and $n_1$) and perturbation $b'(n,E_j,A,n_0,n_1,v,\theta_0)$  such that  for $n_0-v>K_w$ the following holds:

   \begin{description}
     \item[Potential]   for $n_0\leq n \leq n_1$, ${\rm supp}( V)\subset(n_0,n_1)$,  and
     \begin{equation}\label{Jthm141}
        | b'(n,E_j,A,n_0,n_1,v,\theta_0)|\leq \frac{C_w}{n-v}.
     \end{equation}

     \item[Solution for $E_j$] the solution of $(H_0+ b'\mathrm{Id})u=E_ju$ with boundary condition $\theta(n_0,E_j)=\theta_0$ satisfies
     \begin{equation}\label{Jthm142}
        R(n_1,E_j)\leq C_w(\frac{n_1-v}{n_0-v})^{-100} R(n_0,E)
     \end{equation}
     and  for $n_0<n<n_1$,
      \begin{equation}\label{Jthm143}
        R(n,E_j)\leq (\frac{n_1-v}{n_0-v})^{\varepsilon_w} R(n_0,E_j).
     \end{equation}
      \item[Solution for $ {E}_{j^\prime}$ with $j^\prime\neq j$ ]  any  solution of $(H_0+b'\mathrm{Id})u={E}_{j^\prime}u$ satisfies
      for $n_0<n\leq n_1$,
      \begin{equation}\label{Jthm144}
        R(n,{E}_{j^\prime})\leq   (\frac{n_1-v}{n_0-v})^{\varepsilon_w}R(n_0,{E}_{j^\prime}).
     \end{equation}
   \end{description}

On the other hand,  if $N(w)  $  goes to infinity arbitrarily slowly, then  $C_w  $   can also  go  to infinity arbitrarily slowly. 
 Let us in fact choose $C_w$ so that
\begin{equation}\label{Jaddliu}
    C_{w}\geq 4^{N(w+1)}.
\end{equation}

We can also assume for large $w$,
\begin{equation}\label{JGTk}
    T_w\geq 1000^{w}.
\end{equation}
and for large $w$,
\begin{equation*}
    C_w\leq \ln w.
\end{equation*}
Thus eventually, one has
\begin{equation}\label{JGMar3}
  C_{w+1}\leq T_w.
\end{equation}
Let
\begin{equation}\label{JJwdef}
J_w=\sum_{i}^wN(i) T_i .
\end{equation}
By letting $N(w)$ go to infinity arbitrarily slow, we assume
\begin{equation}\label{JGapr91}
 C_w^2N(w)\leq \frac{1}{100} \min_{n\in [J_{w-1},J_w]}h(n),
\end{equation}
where $h(n)$ is given by Theorem \ref{Jthm:InfiniteEmbedded}.


We will also define potential  $b_{n}^\prime$   and $u(n,E_j)$, $j=1,2,\ldots $ on $(0,J_w)$ by induction, such that
\begin{enumerate}[1.]
\item
$u(n,E_j)$ solves  for $n\in (0,J_w)$
\begin{align}\label{Jeigenengj}
     Ju(n,E_j)
     =E_j  u(n,E_j),
\end{align}
and satisfies  boundary condition
\begin{equation}\label{J1boundaryn}
    \frac{u (1,E_j)}{u(0,E_j)}=\tan\theta_j,
\end{equation}
\item
$u(n,E_i)$  for $i=1,2,\cdots, N(w)$ and $w\geq 2$, satisfies
\begin{eqnarray}
R(J_w,E_i)
  &\leq& 2^{N(w)}  N(w)^{50}C_{w}^{-50}R(J_{w-1},E_i).\label{Jeigenj}
\end{eqnarray}

\item
\begin{equation}\label{Jcontrolkr}
    | b_{n}^\prime |\leq 100\frac{N(w)C^2_{w}}{n+1}，
\end{equation}
for $J_{w-1}\leq n\leq J_w$.
\end{enumerate}

By our construction, one has
\begin{eqnarray}
  \frac{J_w}{T_{w+1}} &\leq & 2 \frac{\sum_{i}^wN(i)T_i}{T_{w+1}} \\
  &\leq & 2\frac{N(w) }{C_{w+1}}\sum_{i=1}^w\frac{T_i}{T_w}\\
  &\leq & 4\frac{N(w) }{C_{w+1}}\label{JTk1Jk}.
\end{eqnarray}

The last inequality comes from \eqref{JTwdef} and \eqref{Jaddliu}

 Let  $u(n,E_j)$ be the solution of
\begin{equation}\label{Jdefinu}
   J u=E_ju
\end{equation}
with boundary  condition
\begin{equation*}
    \frac{u(1,E_j)}{u(0,E_j)}=\tan\theta_j.
\end{equation*}


 Now we should show that the $b'$ derived from this construction satisfies the $w+1$-step conditions \eqref{Jeigenengj}-\eqref{Jcontrolkr}. 

 Let us consider $R(n,E_i)$ for $i=1,2,\cdots,N(w+1)$.
 $R(n,E_i)$ decreases  from   point  $J_w+(i-1)T_{w+1}$ to $J_w+iT_{w+1}$, $i=1,2,\cdots,N(w+1)$, and
  may increase from any point $J_w+(m-1)T_{w+1}$ to $J_w+mT_{w+1}$, $m=1,2,\cdots,N(w+1)$ and $m\neq i$.
  That is
  \begin{equation*}
    R(J_w+iT_{w+1},E_i)\leq N^{50}(w)C_{w+1}^{-50} R(J_w+(i-1)T_{w+1},E_{i}),
  \end{equation*}
  and  for $m\neq i$ (see \eqref{Jthm144}),
  \begin{equation}\label{JC^epsilon}
    R(J_w+mT_{w+1},E_i)\leq   C_{w+1}^{\varepsilon_{w+1}}R(J_w+(m-1)T_{w+1},E_{i}),
  \end{equation}
  by  Proposition \ref{JTwocase}.

   Thus by \eqref{JGepsi}, for $i=1,2,\cdots,N(w+1)$,
   \begin{equation*}
    R(J_{w+1},E_i)\leq   N(w)^{50} C_{w+1}^{-50} C_{w+1}^{N(w+1)\varepsilon_w}R(J_{w},E_i)\leq  N(w)^{50} C_{w+1}^{-49} R(J_{w},E_i).
   \end{equation*}

This implies (\ref{Jeigenj}) for $w+1$.
By the construction of $b_n^\prime$, we have
\begin{equation}
|b_n^\prime| <   100\frac{N(w+1)C^2_{w+1}}{n+1},\label{Jb'bound}
\end{equation}
for $J_w\leq n\leq J_{w+1}$.
 This implies \eqref{Jcontrolkr}.

\begin{proof}[\bf Proof of Theorems \ref{Jthm:FiniteEmbedded} and \ref{Jthm:InfiniteEmbedded}]
In the construction of Theorem \ref{Jthm:FiniteEmbedded}, eventually $N(w)$ and $C_w$ are bounded.
In the construction of Theorem \ref{Jthm:InfiniteEmbedded},   $N(w)$ and $C_w$ grow to infinity arbitrarily slowly.
By \eqref{Jcontrolkr} and \eqref{JGapr91}, \eqref{JGgoalb} and \eqref{JGgoala} hold.

It suffices to show that for any $j$, $R(n,E_j) \in \ell^2$.
Below we give the details.

For any $N(w_0-1)<j\leq N(w_0)$, by the construction (see (\ref{Jeigenj})), we have for $w\geq w_0$
\begin{eqnarray}
    R(J_{w+1},E_j)&\leq&  N(w)^{50} C_{w+1}^{-49} R(J_{w},E_j) \nonumber\\
    &\leq&    C_{w+1}^{-25} R(J_{w},E_j) \nonumber \\
    &\leq& T_{w_0}^{25}T_{w+1}^{-25} R(J_{w_0},E_j)\label{Jlastiteration}
\end{eqnarray}
where the second  inequality holds by \eqref{Jaddliu} and the third inequality holds by \eqref{JTwdef}.

By  \eqref{JGepsi}, (\ref{Jthm143}), (\ref{Jthm144}), \eqref{Jaddliu}, \eqref{Jlastiteration} ,\eqref{JC^epsilon} and \eqref{JGMar3}, for all $n\in[J_{w+1},J_{w+2}]$,
\begin{eqnarray}
  R(n,E_j) & \leq&   C_{w+2}^{N_{w+2}\varepsilon_{w+2}}R(J_{w+1},E_j)\nonumber\\
  &\leq&   C_{w+2}^{N_{w+2}\varepsilon_{w+2}}T_{w_0}^{25}T_{w+1}^{-25} R(J_{w_0},E_j)\nonumber\\
     &\leq&
     T_{w_0}^{25}T_{w+1}^{-24} R(J_{w_0},E_j).\label{JlastGx}
\end{eqnarray}
\

Then by  (\ref{JlastGx}), we have
\begin{eqnarray*}
  \sum_{n=J_{w_0+1}}^{\infty} R^2(n,E_j)  &=&\sum_{w\geq w_0}\sum_{n=J_{w+1}}  ^{J_{w+2}} R^2(n,E_j)\\
   &\leq&   \sum_{w\geq w_0}\sum_{n=J_{w+1}}  ^{J_{w+2}}  T_{w_0}^{50}T_{w+1}^{-48} R^2(J_{w_0},E_j) \\
   &\leq& T_{w_0}^{50}R^2(J_{w_0},E_j) \sum_{w\geq w_0} N(w+2)T_{w+2}T_{w+1}^{-48} \\
   &=&  T_{w_0}^{50} R^2(J_{w_0},E_j)\sum_{w\geq w_0}N(w+2)C_{w+2}T_{w+1}^{- 47} \\
     &\leq&  T_{w_0}^{50} R^2(J_{w_0},E_j)\sum_{w\geq w_0}T_{w+1}^{-40} <\infty,
\end{eqnarray*}
since $N(w)$ and $C_w$ go to infinity slowly and $T_w$ satisfies \eqref{JGTk}.
This completes the proof.
\end{proof}

 \section*{Acknowledgments}
    W.L. was supported by the AMS-Simons Travel Grant 2016-2018, NSF DMS-1401204 and  NSF DMS-1700314. D.O . was supported by a Xiamen University Malaysia Research Fund (Grant No: XMUMRF/2018-C1/IMAT/0001). The authors also wish to thank Jake Fillman, Svetlana Jitomirskaya, Milivoje Lukic, Christian Remling, and the anonymous referee for helpful conversations and comments.


\end{document}